\documentclass{amsart}

\usepackage{amsmath,amssymb}
\usepackage{graphicx}
\usepackage{epstopdf}

\newcommand{\arr}{\longrightarrow}
\newcommand{\R}{\mathbb{R}}
\newcommand{\Z}{\mathbb{Z}}

\newcommand{\be}{\mathsf{s}}
\newcommand{\en}{\mathsf{r}}
\newcommand{\G}{\mathfrak{G}}
\newcommand{\symm}{\mathsf{S}}
\newcommand{\alt}{\mathsf{A}}
\newcommand{\full}[1]{\mathsf{F}(#1)}

\newcommand{\F}{\mathcal{F}}
\newcommand{\X}{\mathcal{X}}

\newcommand{\mapdown}[1]%
{\Big\downarrow\rlap{$\vcenter{\hbox{$\scriptstyle#1$}}$}}

\newtheorem{theorem}{Theorem}[section]
\newtheorem{lemma}[theorem]{Lemma}
\newtheorem{proposition}[theorem]{Proposition}
\newtheorem{corollary}[theorem]{Corollary}

\theoremstyle{definition}
\newtheorem{defi}{Definition}[section]
\newtheorem{example}{Example}[section]

\title{Simple groups of dynamical origin}
\author{V. Nekrashevych}

\begin{document}

\maketitle

\begin{abstract}
We associate with every \'etale groupoid $\G$ two normal subgroups
$\symm(\G)$ and $\alt(\G)$ of
the topological full group of $\G$, which are analogs of the symmetric
and alternating groups. We prove that if $\G$ is a minimal groupoid of
germs (e.g., of a group action), then $\alt(\G)$ is simple and is
contained in every non-trivial normal subgroup of the full group. We
show that if $\G$ is expansive (e.g., is the groupoid of germs of
an expansive action of a group), then $\alt(\G)$ is finitely
generated. We also show that $\symm(\G)/\alt(\G)$ is a quotient of
$H_0(\G, \Z/2\Z)$.
\end{abstract}

\section{Introduction}

We associate with every group or semigroup of local homeomorphisms of
the Cantor set (more generally with an \'etale groupoid with space of
units homeomorphic to the Cantor set) groups $\alt(\G)$ and $\symm(\G)$ that are analogs
of the alternating and symmetric groups. They are subgroups of the
\emph{topological full group} of the action. We prove that if the
action is minimal, i.e., if all orbits are dense, then $\alt(\G)$ is
simple. This produces a very wide class of simple groups, that
includes some well known examples: block-diagonal direct limits of
direct products of finite alternating groups, derived subgroups of topological full groups of minimal
homeomorphisms of the Cantor set, derived subgroups of the
Higman-Thompson groups and full groups of one-sided shifts of finite
type~\cite{matui:fullonesided}, and many interesting new examples.
We show that if the action is expansive, then $\alt(\G)$ is
finitely generated. In a subsequent paper we use these results to
construct first examples of simple groups of intermediate growth.

Topological full groups (of $\Z$-actions) were defined by T.~Giordano,
I.~F.~Putnam, and C.~F.~Skau in connection with the theory of orbit
equivalence of minimal homeomorphisms of the Cantor set,
see~\cite{gior:full}. They also appeared earlier (for a special class of group actions) in the paper~\cite{krieger:homgroups}.
For a homeomorphism
$\tau:\mathcal{X}\arr\mathcal{X}$,
the topological full group $[[\tau]]$ is the group of all
homeomorphisms $\phi:\mathcal{X}\arr\mathcal{X}$ such that for every
$x\in\mathcal{X}$ there exists $n\in\Z$ and a neighborhood $U$ of $x$
such that $\phi|_U=\tau^n|_U$.

The definition is local in the sense that $[[\tau]]$ depends only on the germs of the iterations of
$\tau$. The notion of a topological full group has therefore a very natural generalization
to arbitrary semigroups of local homeomorphisms, and a natural setting for such
a generalization is the theory of  \'etale groupoids. A group $G$ of homeomorphisms
of a topological space $\X$ is a full topological group of some semigroup of local
homeomorphisms if the following condition is satisfied. If
$\phi:\X\arr\X$ is a homeomorphism, and for every $x\in\X$ there
exists a neighborhood $U$ of $x$ such that $\phi|_U$ coincides with
the restriction of an element of $G$ to $U$, then $\phi$ belongs to $G$.
There are many examples of topological full groups in this sense: the Thompson
groups~\cite{intro_tomp}, a version of the Y.~Lodha and J.~Moore’s
example~\cite{lodhamoore}, groups associated
with expanding maps~\cite{nek:fpresented}, diagram groups~\cite{gubasapir:diagram}, etc.

Topological full groups of \'etale groupoids were studied by H.~Matui in a series
of papers~\cite{matui:fullI,matui:etale,matui:fullonesided}.
In particular, he showed that for some classes of groupoids the
derived subgroup of the topological full group is simple, and that in some cases
(e.g., for minimal subshifts) it is also finitely generated, or even finitely presented
(for groupoids associated with one-sided shifts of finite type). For a survey of results on full topological groups, see~\cite{cornulier:pleinstopologiques}.

It was shown in~\cite{ChJN} that if an action of the free abelian group $\Z^n$ on a Cantor set is minimal and expansive (i.e., conjugate to a subshift), then the derived subgroup of the topological full group is finitely generated.

We prove the following generalizations (see Theorems~\ref{th:simple} and~\ref{th:finitelygenerated} below) of the above mentioned results. They are valid for a much wider class of groupoids, though we leave open the question when the group $\alt(\G)$ coincides with the derived subgroup of the topological full group.

\begin{theorem}
Let $\G$ be a minimal groupoid of germs. Then $\alt(\G)$ is simple,
and is contained in every non-trivial normal subgroup of the
topological full group of $\G$.
\end{theorem}

\begin{theorem}
If $\G$ is expansive and has infinite orbits, then $\alt(\G)$ is finitely generated.
\end{theorem}

The expansivity condition for groupoids is a natural generalization of the notion
of an expansive dynamical system. In particular, the groupoid of an expansive
action of a group on a Cantor set is expansive.

We illustrate these results by defining groups naturally associated
with quasicrystals (e.g., Penrose tilings). We associate with every
repetitive aperiodic Delaunay set with finite local complexity a
finitely generated simple group, see Theorem~\ref{th:quasicrystals}. This is a generalization of a result of~\cite{ChJN} about the Penrose tiling.

We also study the quotient $\symm(\G)/\alt(\G)$. We prove that it is an
image of the zero homology group $H_0(\G, \Z/2\Z)$ under a
natural homomorphism.

In all known to us examples, $\alt(\G)$ coincides with the derived
subgroup $\mathsf{D}(\G)$ of the topological full group, while $\symm(\G)$ coincides
with the kernel $\mathsf{K}(\G)$ of the natural \emph{index map} from the topological
full group to $H_1(\G, \Z)$. In particular, it follows from the results of
H.~Matui, that this is true for \emph{almost finite} and for \emph{purely infinite}
groupoids. If $\G$ is such that $\alt(\G)=\mathsf{D}(\G)$ and
$\symm(\G)=\mathsf{K}(\G)$, then our results confirm for $\G$ a
conjecture of H.~Matui about the structure of the
abelianization of the topological full group (see~\cite[Conjecture~2.9]{matui:product}). It would be interesting to understand when the equalities $\alt(\G)=\mathsf{D}(\G)$ and $\symm(\G)=\mathsf{K}(\G)$ hold.

\subsection*{Acknowledgments.}  The author is grateful for discussions and remarks to the first version of the paper by M.~Lawson, N.~Matte~Bon, H.~Matui, and K.~Medynets.

\section{Etale groupoids}

\subsection{Basic definitions}
A \emph{groupoid} is a small category of isomorphisms,
i.e., a set $\G$ with partially defined multiplication
$\gamma_1\gamma_2$ and everywhere defined involutive operation of
taking inverse $\gamma^{-1}$ satisfying the following axioms.
\begin{enumerate}
\item If $\gamma_1\gamma_2$ and $(\gamma_1\gamma_2)\gamma_3$ are
  defined, then $\gamma_2\gamma_3$ and $\gamma_1(\gamma_2\gamma_3)$
  are defined and
  $(\gamma_1\gamma_2)\gamma_3=\gamma_1(\gamma_2\gamma_3)$.
\item The products $\gamma\gamma^{-1}$ and $\gamma^{-1}\gamma$ are always
  defined. If $\gamma_1\gamma_2$ is defined, then
  $\gamma_1=\gamma_1\gamma_2\gamma_2^{-1}$ and $\gamma_2=\gamma_1^{-1}\gamma_1\gamma_2$.
\end{enumerate}

A \emph{topological groupoid} is a groupoid together with a topology on it such that the
operations of multiplication and taking inverse are continuous. We \emph{do not} require
that the groupoid is Hausdorff.

The elements of the form $\gamma\gamma^{-1}$ are called
\emph{units}. We denote the set of units of a groupoid
$\G$ by $\G^{(0)}$, the set of composable pairs by $\G^{(2)}$, the source
and range maps by $\be$ and $\en$, so that
\[\be(\gamma)=\gamma^{-1}\gamma,\qquad\en(\gamma)=\gamma\gamma^{-1}\]
for all $\gamma\in\G$.

Units $x, y\in\G^{(0)}$ belong to one \emph{$\G$-orbit} if there exists
$\gamma\in\G$ such that $\be(\gamma)=x$ and $\en(\gamma)=y$.
A groupoid $\G$ is said to be
\emph{minimal} if all its orbits are dense in $\G^{(0)}$.

The \emph{isotropy group} of a unit $x\in\G^{(0)}$ is the group
$\{\gamma\in\G\;:\;\be(\gamma)=\en(\gamma)=x\}$.
The groupoid $\G$ is said to be \emph{principal} if all its isotropy
groups are trivial, i.e., if $\be(\gamma)=\en(\gamma)=x$ implies
$\gamma=x$.

The groupoid is said to be \emph{essentially principal} if the set of
points with trivial isotropy group is dense in $\G^{(0)}$.

For $U\subset\G^{(0)}$ we denote by $\G|_U$ the subgroupoid
$\{\gamma\in\G\;:\;\be(\gamma), \en(\gamma)\in U\}$, called the
\emph{restriction} of $\G$ to $U$. Note that the orbits of $\G_U$ are
equal to the intersections of the orbits of $\G$ with $U$.

\begin{example}
\label{examp:action}
Let $G$ be a group acting by homeomorphisms on a topological space $\mathcal{X}$. Then
$G\times\mathcal{X}$ is a groupoid in a natural way. Its elements
are multiplied according to the rule
\[(g_1, x_1)(g_2, x_2)=(g_1g_2, x_2),\]
where the product is defined if and only if $g_2(x_2)=x_1$. We have
$\be(g, x)=x$ and $\en(g, x)=g(x)$, where a point $x\in\mathcal{X}$ is
identified with $(1, x)$. This groupoid is called the
\emph{groupoid of the action}.
\end{example}

\begin{example}
\label{examp:germs}
Let $G$ be a discrete group acting by homeomorphisms on $\X$.
Introduce an equivalence relation on $G\times\X$ by the rule: $(g_1, x_1)\sim (g_2, x_2)$ if
$x_1=x_2$ and there exists a neighborhood $U$ of $x_1$ such that
$g_1|_U=g_2|_U$. The equivalence relation $\sim$ agrees with the
multiplication in the groupoid of the action, so that the quotient
$G\times\mathcal{X}/\sim$ is also a groupoid. We call it the
\emph{groupoid of germs} of the action. The $\sim$-equivalence classes
are called \emph{germs}. For every germ $(g, x)\in\G$ and every
neighborhood $U$ of $x$ the set of germs $(g, y)$ for $y\in U$
is considered to be an open set of $\G$, and the
collection of all such open sets is a basis of topology on the groupoid of germs.
\end{example}

\noindent
\textbf{Standing assumptions.} All groupoids in this paper are
\'etale (see below) and the
unit space $\G^{(0)}$ is homeomorphic to the Cantor set. Some of our definition (e.g., bisection,
compact generation) are adapted specifically for the case of Cantor
set space of units, and should be changed in the general case.

\begin{defi}
A \emph{$\G$-bisection} is a compact open subset $F\subset\G$ such that $\be:F\arr\be(F)$
and $\en:F\arr\en(F)$ are homeomorphisms.

A topological groupoid $\G$ is said to be \emph{\'etale} if it has a basis of
topology consisting of $\G$-bisections.
\end{defi}

Every bisection $F\subset\G$ defines the \emph{associated
  homeomorphism} $\be(F)\arr\en(F)$ given by the formula
$x\mapsto\en(Fx)$ for $x\in\be(F)$, or, equivalently, by
$\be(\gamma)\mapsto\en(\gamma)$ for $\gamma\in F$.

The groupoid of an action (see Example~\ref{examp:action}) and the
groupoid of germs of an action (see Example~\ref{examp:germs}) are \'etale.

\begin{defi}
An \'etale groupoid $\G$ is said to be a \emph{groupoid of germs} (or \emph{effective}) if
for every non-unit element $\gamma\in\G$ and every bisection $F$
containing $\gamma$, there exists $\delta\in F$ such that
$\be(\delta)\ne\en(\delta)$.
\end{defi}

If $\G$ is a groupoid of germs, then for every element $\gamma$ and
for every bisection $F\ni\gamma$, the element $\gamma$ is
uniquely determined by $\be(\gamma)$ and the homeomorphism associated
with $F$.

For every \'etale groupoid $\G$ the groupoid of germs of the local
homeomorphisms of $\G^{(0)}$ defined by bisections is a groupoid of
germs which is a quotient of $\G$. We call this quotient \emph{the
groupoid of germs associated with $\G$}.

There is a close connection between groupoids of germs and essentially
principal groupoids.

\begin{proposition}
\label{pr:essprincipalHaus}
If $\G$ is essentially principal and Hausdorff, then it is a groupoid
of germs.

Every second countable groupoid of germs is essentially principal.
\end{proposition}

\begin{proof}
Assume that $\G$ is essentially principal and Hausdorff. Suppose that
$\G$ is not a groupoid of germs. Then there exists
$\gamma\in\G\setminus\G^{(0)}$ and a bisection $F\ni\gamma$ such that
for every $\delta\in F$ we have $\be(\delta)=\en(\delta)$. Since $\G$
is essentially principal, for every bisection
$F'\subset F$, the set $F'$ contains a unit. It follows that
the unit $\be(\gamma)$ and the element $\gamma$ do not have disjoint
neighborhoods, which contradicts Hausdorffness of $\G$.

Suppose that $\G$ is a second countable groupoid of germs, and let
$F\subset\G$ be a bisection. Consider the set
$A=\{\gamma\;:\;\be(\gamma)=\en(\gamma)\}$. The set $A$ is closed. According to the
definition of a groupoid of germs, if $\gamma$ belongs to the interior of
$A$, then $\gamma$ is a unit. It follows that the set of non-unit
elements of $F$ is nowhere dense. Consequently, by Baire Category
Theorem, the set of units with trivial isotropy group is dense in
$\G^{(0)}$, if $\G$ is second countable.
\end{proof}

\begin{defi}
Let $\G$ be an \'etale groupoid. Its \emph{full group}
(\emph{topological full group}), denoted $\full\G$,
is the set of bisections $F\subset\G$ such that $\be(F)=\en(F)=\G^{(0)}$.
\end{defi}

The topological full group is often denoted in the literature by $[[\G]]$, but we changed the
notation for consistency with our notations for various normal subgroups
of the full group.

Our definition is slightly different from the definition due to
H.~Matui~\cite{matui:etale,matui:fullonesided}, where
the topological full group is defined as the topological full
group of the groupoid of germs associated with $\G$, so that it is a
homeomorphism group of $\G^{(0)}$. But if $\G$ is a
groupoid of germs, then our definition agrees with the H.~Matui's
definition. Note that it is often assumed in some papers that $\G$ is
essentially principal and Hausdorff, which implies that $\G$ is a
groupoid of germs, see Proposition~\ref{pr:essprincipalHaus}.

\subsection{Homology of \'etale groupoids}

Following H.~Matui, let us define, for an \'etale groupoid $\G$ and an
abelian group $A$ the homology groups $H_k(\G, A)$.

Denote by $\G^{(n)}$ the space of
sequences $(\gamma_1, \gamma_2, \ldots, \gamma_n)\in\G^n$ such that
the product $\gamma_1\gamma_2\ldots\gamma_n$ is defined. In
particular, $\G^{(1)}=\G$. We keep the
notation $\G^{(0)}$ for the space of units.

For every compact open subset $F\subset\G^{(n)}$ and $a\in A$, denote
by $a_F$ the function $\G^{(n)}\arr A$ equal to $a$ on $F$ and to $0$
on $\G^{(n)}\setminus F$. Denote by $C_c(\G^{(n)}, A)$ the
sub-group of the abelian group $A^{\G^{(n)}}$ generated by
the functions of the form $a_F$. If $\G$ is Hausdorff, then
$C_c(\G^{(n)}, A)$ coincides with the group of continuous $A$-valued compactly
supported functions on $\G^{(n)}$, where $A$ has discrete topology.

Consider the maps $d_i:\G^{(n)}\arr\G^{(n-1)}$, for $i=0, 1, \ldots, n$,  given by
\[d_i(\gamma_1, \gamma_2, \ldots, \gamma_n)=\left\{\begin{array}{ll}
    (\gamma_2, \gamma_3, \ldots, \gamma_n) & \text{for $i=0$}\\
(\gamma_1, \ldots, \gamma_i\gamma_{i+1}, \ldots, \gamma_n) & \text{for
  $1\le i\le n-1$}\\ (\gamma_1, \gamma_2, \ldots, \gamma_{n-1}) &
\text{for $i=n$.}\end{array}\right.\]
If $n=1$, then we set $d_0(\gamma)=\be(\gamma)$ and $d_1(\gamma)=\en(\gamma)$.

Define $d_{i*}:C_c(\G^{(n)}, A)\arr C_c(\G^{(n-1)}, A)$ by
\[d_{i*}(f)(x)=\sum_{y\in d_i^{-1}(x)}f(y),\]
and
\[\delta_n=\sum_{i=0}^n(-1)^id_{i*}.\]

Then
\[0\stackrel{\delta_0}{\longleftarrow} C_c(\G^{(0)}, A)
\stackrel{\delta_1}{\longleftarrow} C_c(\G^{(1)}, A)
\stackrel{\delta_2}{\longleftarrow} C_c(\G^{(2)}, A)\cdots\]
is a chain complex. We denote by $H_n(\G,
A)=\mathrm{Ker}\delta_n/\mathrm{Im}\delta_{n+1}$ its homology.

\begin{defi}
The \emph{index map} is the homomorphism $I:\full\G\arr H_1(\G, \Z)$
mapping a bisection $F\in\full\G$ to the class $[1_F]$ of its indicator.
\end{defi}

It is easy to check that the index map is a well defined homomorphism,
see~\cite{matui:fullonesided,matui:etale}.
We denote by $\mathsf{K}(\G)$ the kernel of the index map, and by
$\mathsf{D}(\G)$ the derived subgroup $[\full\G, \full\G]$ of the full
group. We obviously have $\mathsf{D}(\G)\le\mathsf{K}(\G)$.

\section{Groups $\symm(\G)$ and  $\alt(\G)$}

\subsection{Multisections}

\begin{defi}
A \emph{multisection of degree $d$} is a collection of $d^2$
bisections $\{F_{i, j}\}_{i, j=1}^d$ such that
\begin{enumerate}
\item $F_{i_2, i_3}F_{i_1, i_2}=F_{i_1, i_3}$ for all $1\le i_1, i_2,
  i_3\le d$;
\item the bisections $F_{i, i}$ are disjoint subsets of $\G^{(0)}$.
\end{enumerate}
We call $\bigcup_{i=1}^dF_{i, i}$ the \emph{domain} of the
multisection, and the sets $F_{i, i}$ the \emph{components of the domain}.
\end{defi}

See Figure~\ref{fig:multisection} for a schematic illustration of a
degree 3 multisection. In semigroup-theoretic terms we can define multisections as homomorphism from a finite symmetric inverse monoid to the inverse monoid of bisections of $\G$.

\begin{figure}
\centering
\includegraphics{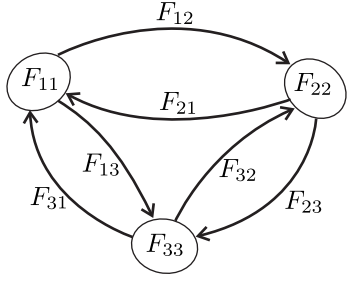}
\caption{A multisection}
\label{fig:multisection}
\end{figure}

\begin{lemma}
\label{lem:smallmsections}
Let $X=\{x_1, x_2, \ldots, x_d\}\subset\G^{(0)}$ be a set contained in
one $\G$-orbit, and let $U\subset\G^{(0)}$ be a neighborhood of
$X$. Then there exists a multisection $\F=\{F_{i, j}\}_{i, j=1}^d$
such that $x_i\in F_{i, i}$, $x_j=\en(F_{i, j}x_i)$, and the domain of $\F$ is
contained in $U$.
\end{lemma}

\begin{proof}
Find for every $i=2, \ldots, d$ an element $\gamma_i\in\G$ such that
$\be(\gamma_i)=x_1$ and $\en(\gamma_i)=x_i$. Set $\gamma_1=x_1$. Find
$\G$-bisections $H_i$ such that $\gamma_i\in H_i$. We can find a clopen
neighborhood $W$ of $x_1$ such that $W\subset\be(H_i)\cap U$ for every
$i$, and the sets $\en(H_iW)$ are disjoint subsets of $U$. We also assume
that $H_1W=W$. Define then $F_{i, j}=(H_jW)(H_iW)^{-1}$.
\end{proof}

Let $\F=\{F_{i, j}\}_{i, j=1}^d$ be a multisection with the domain $U$, and let $\pi\in
\symm_d$ be a permutation, where $\symm_d$ denotes the symmetric group of degree $d$. Then $\bigcup_{i=1}^dF_{i,
  \pi(i)}\cup(\G^{(0)}\setminus U)$ is an element of $\full\G$, which we
will denote $\F_\pi$. It is easy to see that $\pi\mapsto\F_\pi$ is an
embedding of $\symm_d$ into $\full\G$. Let us denote by $\symm(\F)$ the
image of this embedding, and denote by $\alt(\F)$ the image of the
alternating subgroup $\alt_d<\symm_d$.

\begin{proposition}
\label{pr:covering}
Let $\F=\{F_{i, j}\}_{i, j=1}^d$ be a multisection of degree $d\ge
5$. Let $\F^{(k)}=\{F_{i, j}^{(k)}\}_{i, j=1}^d$, for $k=1, 2, \ldots,
n$, be multisections such that $\bigcup_{k=1}^nF_{i,
  j}^{(k)}=F_{i, j}$ for all $1\le i, j\le d$.
Then $\alt(\F)$ is contained in the group generated by $\bigcup_{k=1}^n\alt(\F^{(k)})$.
\end{proposition}

\begin{proof}
Let us prove at first the following lemma.

\begin{lemma}
\label{lem:pipi1}
Let $d\ge 5$. Then the set $A=\{(\pi, \pi,
1)\;:\;\pi\in\alt_d\}\cup\{(1, \pi, \pi)\;:\;\pi\in\alt_d\}$ generates
$\alt_d\times\alt_d\times\alt_d$.
\end{lemma}

\begin{proof}
We have $[(\pi_1, \pi_1, 1), (1, \pi_2, \pi_2)]=(1, [\pi_1, \pi_2],
1)$. Since $\alt_d$ is equal to its commutator subgroup, every element
$(1, \pi, 1)$, $\pi\in\alt_d$, belongs to $\langle A\rangle$. It
follows that every element of the form $(\pi, 1, 1)$ or $(1, 1, \pi)$
belongs to $\langle A\rangle$.
\end{proof}

It is enough to prove Proposition~\ref{pr:covering} for $n=2$.

Denote $P_{i, j}=F_{i, j}^{(1)}\cap F_{i, j}^{(2)}$, $D_{i,
  j}^{(1)}=F_{i, j}^{(1)}\setminus F_{i, j}^{(2)}$, $D_{i,
  j}^{(2)}=F_{i, j}^{(2)}\setminus F_{i, j}^{(1)}$. Then
$\mathcal{P}=\{P_{i, j}\}_{i, j=1}^d$, $\mathcal{D}^{(1)}=\{D_{i,
  j}^{(1)}\}_{i, j=1}^d$, and $\mathcal{D}^{(2)}=\{D_{i,
  j}^{(2)}\}_{i, j=1}^d$  are multisections. It is obvious that
$\alt(\F)<\langle\alt(\mathcal{P})\cup\alt(\mathcal{D}^{(1)})\cup\alt(\mathcal{D}^{(2)})\rangle$,
hence it is enough to prove that $\alt(\mathcal{P})$,
$\alt(\mathcal{D}^{(1)})$, and $\alt(\mathcal{D}^{(2)})$ are subgroups
of $\langle\alt(\F^{(1)})\cup\alt(\F^{(2)})\rangle$. The elements of
$\alt(\F^{(i)})$ leave the domains of $\mathcal{P}$ and
$\mathcal{D}^{(i)}$ invariant. For every $\pi\in\alt_d$, restrictions
of $\F_\pi^{(1)}$ to the domains of $\mathcal{P}$ and
$\mathcal{D}^{(1)}$ are equal to $\mathcal{P}_\pi$ and
$\mathcal{D}^{(1)}_\pi$, respectively, while its restriction to the domain of
$\mathcal{D}^{(2)}$ is trivial. Similarly, restrictions of
$\F_\pi^{(2)}$ to the domains of $\mathcal{P}$ and $\mathcal{D}^{(2)}$
are equal to $\mathcal{P}_\pi$ and $\mathcal{D}^{(2)}_\pi$, respectively, while its
restriction to the domain of $\mathcal{D}^{(1)}$ is trivial. It
follows from Lemma~\ref{lem:pipi1} that $\alt(\mathcal{P})$ and
$\alt(\mathcal{D}^{(i)})$, $i= 1, 2$, belong to $\langle\alt(\F^{(1)})\cup\alt(\F^{(2)})\rangle$.
\end{proof}

Let $\F=\{F_{i, j}\}_{i, j=1}^d$ be a multisection, and let $U\subset
F_{i, i}$ be a clopen set, for some $i\in\{1, 2, \ldots, d\}$.
Denote, for every $j=1, 2, \ldots, d$, by
$U_j$ the set $\en(F_{i, j}U)\subset F_{j, j}$, and by $F_{i, j}'$ the
restriction $F_{i, j}U_i$. Then $\F'=\{F_{i, j}'\}_{i, j=1}^d$ is a
multisection, which we call \emph{restriction} of $\F$ to $U$.

\begin{proposition}
\label{pr:unionintersection}
Let $\mathcal{G}=\{G_{i, j}\}_{i, j=1}^{d_1}$ and $\mathcal{H}=\{H_{i, j}\}_{i,
  j=1}^{d_2}$ be multisections such that $d_i\ge 3$. Suppose that the
intersection $U$ of the domains of $\mathcal{G}$ and $\mathcal{H}$ is
equal to the intersection $G_{1, 1}\cap H_{1, 1}$. Denote by $\mathcal{G}'$
and $\mathcal{H}'$ the restrictions to $U$ of $\mathcal{G}$ and $\mathcal{H}$,
respectively. Let $\F$ be the multisection with the domain equal to
the union of the domains of $\mathcal{G}'$ and $\mathcal{H}'$ and
consisting of all possible compositions
of the elements of $\mathcal{G}'\cup\mathcal{H}'$. We have then
$\alt(\F)\le\left\langle\alt(\mathcal{G})\cup\alt(\mathcal{H})\right\rangle$.
\end{proposition}

\begin{proof}
Let us prove at first the following lemma.

\begin{lemma}
Let $X, Y$ be finite sets such that $|X|, |Y|\ge 3$, and $|X\cap
Y|=1$. Then the normal closure in the group $\langle
\alt_X\cup\alt_Y\rangle$ of the set $[\alt_X, \alt_Y]$ is the whole
alternating group $\alt_{X\cup Y}$.
\end{lemma}

\begin{proof}
Denote by $t$ the unique element of $X\cap Y$.
Let us show at first that $\alt_X\cup\alt_Y$ generates $\alt_{X\cup
  Y}$. It is known that $\alt_{X\cup Y}$ is generated by cycles of length
three. If a cycle $(x, y, z)$ is such that $\{x, y, z\}\subset X$ or
$\{x, y, z\}\subset Y$, then it already belongs to
$\alt_X\cup\alt_Y$. If the cycle contains $t$,
then it also belongs to $\langle\alt_X\cup\alt_Y\rangle$,
since a cycle $(y_1, x_1, t)$ for $x_1\in X$ and $y_1\in Y$
is equal to the commutator of the cycles $(x_1, x_2, t)\in\alt_X$ and
$(y_1, y_2, t)\in\alt_Y$.

Suppose that $\{x, y, z\}$ is such that $x, y\in X\setminus Y$ and $z\in
Y\setminus X$. Then $(x, y, z)$ is equal to the cycle obtained by
conjugating $(x, y, t)\in \alt_X$ by a cycle of the form $(t, z, t')\in\alt_Y$ such that
$\{t, z, t'\}\in Y$.

We have shown that $\langle\alt_X\cup\alt_Y\rangle=\alt_{X\cup Y}$.
As we have mentioned above, any cycle of the form $(y_1, x_1, t)$ for
$x_1\in X$ and $y_1\in Y$ belongs to $[\alt_X, \alt_Y]$. Conjugating it
by elements of $\alt_{X\cup Y}$, we obtain all cycles of length
three, thus generating $\alt_{X\cup Y}$ by conjugates of a single
element of $[\alt_X, \alt_Y]$.
\end{proof}

It is easy to see that for any $\pi_i\in\alt_{d_i}$ we have $[\mathcal{G}_{\pi_1},
\mathcal{H}_{\pi_2}]\in\alt(\F)$. If we
conjugate an element of $\alt(\F)$ by an element of the set
$\alt(\mathcal{G})\cup\alt(\mathcal{H})$, then we get
an element of $\alt(\F)$. It follows that the normal closure of the
set $[\alt(\mathcal{G}), \alt(\mathcal{H})]$ in the group
$\left\langle\alt(\mathcal{G})\cup\alt(\mathcal{H})\right\rangle$ is
contained in $\alt(\F)$. The above lemma finishes then the proof.
\end{proof}

\subsection{Definition of $\alt(\G)$ and $\symm(\G)$}

\begin{defi}
Denote by $\symm_d(\G)$ (resp.\ $\alt_d(\G)$) the subgroup of $\full\G$ generated by the union of the subgroups $\symm(\F)$ (resp.\ $\alt(\F)$) for all multisections $\F$ of degree $d$.

Denote $\symm(\G)=\symm_2(\G)$ and $\alt(\G)=\alt_3(\G)$.
\end{defi}

Note that for every multisection $\F=\{F_{i, j}\}_{i, j=1}^d$, and
every $g\in\full\G$, the set $\F^g=\{g^{-1}F_{i, j}g\}_{i, j=1}^d$ is
also a multisection, and $\symm(\F^g)=g^{-1}\symm(\F)g$,
$\alt(\F^g)=g^{-1}\alt(\F)g$.
It follows that all subgroups $\alt_d(\G)$ and $\symm_d(\G)$ are a normal in $\full\G$.

We obviously have $\symm_d(\G)\ge\symm_{d+1}(\G)$ for every $d\ge 2$, and $\alt_d(\G)\ge\alt_{d+1}(\G)$ for every $d\ge 3$. Namely, for every
multisection $\F$ of degree $d+1$ the group $\symm(\F)$ (resp.\ the group $\alt(\F)$) is generated by
the groups $\symm(\F')$ (resp.\ $\alt(\F')$) where $\F'$ runs through sub-multisections of $\F$ of degree $d$.

In particular, $\alt(G)\ge\alt_d(\G)$ for all $d\ge 3$, and $\symm(\G)\ge\symm_d(\G)$ for all $d\ge 2$.

Note that if there exists a  $\G$-orbit of size $d\ge 3$, then $\alt_d(\G)$ and $\symm_d(\G)$ act on it transitively, while $\alt_{d+1}(\G)$ and $\symm_{d+1}(\G)$ act on it trivially, hence $\alt_d(\G)\ne\alt_{d+1}(\G)$ and $\symm_d(\G)\ne\symm_{d+1}(\G)$. It follows that $\alt_{d+1}(\G)\ne\alt(\G)$ and $\symm_{d+1}(\G)\ne\symm(\G)$ in this case.

\begin{proposition}
\label{pr:anydegreegen}
Suppose that all orbits of $\G$ have at least $n$ points. Then $\alt_d(\G)=\alt(\G)$ for all $3\le d\le n$, and $\symm_d(\G)=\symm(\G)$ for all $2\le d\le n$.
\end{proposition}

\begin{proof} Let us prove the proposition for $\symm(\G)$, the proof
  for $\alt(\G)$ is analogous.
We know that $\symm_k(\G)\ge\symm_{k+1}(\G)$ for all $2\le k\le n-1$.
We have to prove that $\symm_k(\G)\le\symm_{k+1}(\G)$. Let $\F$ be
a multisection of degree $k$. It is shown the same way as in the proof of Lemma~\ref{lem:smallmsections} that
for any sufficiently small subset $U$ of a component of the domain of
$\F$ there exists a multisection $\F'$ of degree $k+1$ containing
restriction of $\F$ to $U$. It follows that we can split a component
$F_{1, 1}$ of the domain of $\F$ into a finite disjoint union of sets
$U_i$ such that the restriction of $\F$ to every $U_i$ is contained in a
multisection of degree $k+1$. It follows that $\symm(\F)\subset\symm_{k+1}(\G)$.
\end{proof}

\begin{corollary}
\label{cor:anydegreegen}
Suppose that every orbit of $\G$ is infinite. Then $\alt_d(\G)=\alt(\G)$ for all $d\ge 3$, and $\symm_d(\G)=\symm(\G)$ for all $d\ge 2$.
\end{corollary}

It is easy to see that $\symm(\G)\le\mathsf{K}(\G)$  and
$\alt(\G)\le\mathsf{D}(\G)$. We have a diagram of
subgroups of $\full\G$ shown on Figure~\ref{fig:poset}.

\begin{figure}
\centering
\includegraphics{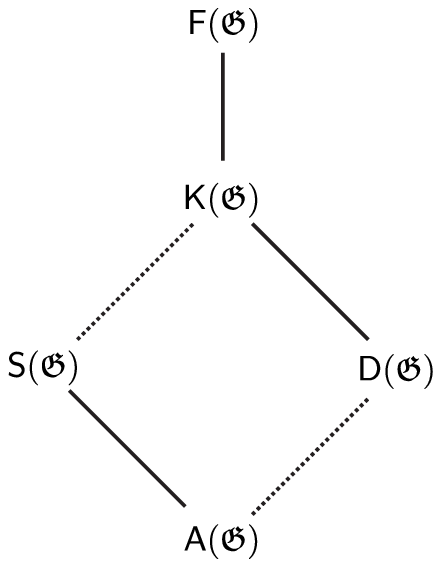}
\caption{Subgroups of $\full\G$}
\label{fig:poset}
\end{figure}

\subsection{Rigidity}

The next theorem follows directly from  Theorems~0.2 and~3.3 of~\cite{rubin:reconstr}.

Let $G$ be a group acting by homeomorphisms of a topological space $\X$. Denote, for $U\subset\X$, by $G_{(U)}$ the subgroup consisting of elements $g\in G$ acting trivially on $\X\setminus U$.
We say that $G$ is \emph{locally minimal} if there exists a basis of open sets $\mathcal{U}$ such that for every $U\in\mathcal{U}$ and $x\in U$ the $G_{(U)}$-orbit of $x$ is dense in $U$. 

\begin{theorem}
\label{th:rubin}
If $G_i$ are locally minimal groups of homeomorphisms of locally compact Hausdorff spaces $\X_i$, then for every isomorphism $\phi:G_1\arr G_2$ there exists a homeomorphism $F:\X_1\arr\X_2$ such that $\phi(g)=F\circ  g\circ F^{-1}$ for all $g\in G_1$.
\end{theorem}

As a corollary, we get the following.

\begin{theorem}
Let $\G_i$, for $i=1, 2$, be a minimal groupoid of germs. Then the following conditions are equivalent.
\begin{enumerate}
\item The groupoids $\G_1$ and $\G_2$ are isomorphic as topological
  groupoids.
\item The groups $\full{\G_1}$ and $\full{\G_2}$ are isomorphic.
\item The groups $\mathsf{D}(\G_1)$ and $\mathsf{D}(\G_2)$ are
  isomorphic.
\item The groups $\mathsf{K}(\G_1)$ and $\mathsf{K}(\G_2)$ are
  isomorphic.
\item The groups $\symm(\G_1)$ and $\symm(\G_2)$ are isomorphic.
\item The groups $\alt(\G_1)$ and $\alt(\G_2)$ are isomorphic.
\end{enumerate}
\end{theorem}

Equivalence of statements (1)--(4) was proved in~\cite{matui:fullonesided}. For methods of
proving similar statements without using M.~Rubin's theorem, 
see~\cite{matui:fullonesided,medynets:reconstruction}.

\section{Simplicity of $\alt(\G)$}

\begin{theorem}
\label{th:simple}
Suppose that $\G$ is a minimal groupoid of germs. Then every
non-trivial subgroup of $\full\G$ normalized by $\alt(\G)$ contains
$\alt(\G)$. In particular, $\alt(\G)$ is simple and is contained in
every non-trivial normal subgroup of $\full\G$. 
\end{theorem}

\begin{proof}
Let $U\subset\G^{(0)}$ be a clopen subset. We will naturally identify
$\alt(\G|_U)$ with a subgroup of $\alt(\G)$, extending elements of
$\alt(\G|_U)$ identically on $\G^{(0)}\setminus U$.

Let $N\le\full\G$ be a subgroup normalized by $\alt(\G)$,
and let $g\in N$ be a non-trivial element. Since
$\G$ is a groupoid of germs, there exists an element $\gamma\in g$
such that $\be(\gamma)\ne\en(\gamma)$. There
exists a clopen neighborhood $U$ of $\be(\gamma)$ such that $U\cap
g(U)=\emptyset$. Let $\F$ be a multisection of degree $\ge 3$ whose
domain is a subset of $U$. Let $h_1, h_2\in\alt(\F)$. Then $g\F g^{-1}$
is a multisection with domain in $g(U)$, and $gh_1g^{-1}\in\alt(g\F
g^{-1})$. It follows that $[g^{-1}, h_1]=gh_1^{-1}g^{-1}h_1$ acts trivially outside
$U\cup g(U)$, as $h_1$ on $U$, and as $gh_1^{-1}g^{-1}$ on
$g(U)$. Consequently, $[[g^{-1}, h_1], h_2]=[h_1, h_2]$. But
$[[g^{-1}, h_1], h_2]\in N$ since $N$ is normalized by $\alt(\G)$,
$h_1, h_2\in\alt(\F)\subset\alt(\G)$ and $g\in N$. It
follows that the derived subgroup of $\alt(\F)$ is contained in
$N$. But every element of $\alt(\F)$ is a commutator of elements of
$\alt(\F)$, hence the derived subgroup of $\alt(\F)$ is $\alt(\F)$, so
that we have $\alt(\F)\le N$.

We have proved that there exists a non-empty clopen subset $U$ of
$\G^{(0)}$ such that $\alt(\G|_U)\le N$.

It remains to prove that the normal closure of $\alt(\G|_U)$ in $\alt(\G)$
is equal to $\alt(\G)$. It is enough to show that for every multisection $\F$ of
degree 5, the group $\alt(\F)$ is contained in the normal closure of $\alt(\G|_U)$ in
$\alt(\G)$, see Proposition~\ref{pr:anydegreegen}.

Let $\F=\{F_{i, j}\}_{i, j=1}^5$ be a multisection. Let $x_1\in F_{1,
  1}$ be arbitrary, and denote $x_i=\en(F_{1, i}x_1)$. Then all $x_i\in
F_{i, i}$ belong to one $\G$-orbit. Since $\G$ is minimal, all its
orbits are infinite and dense, and we can find elements
$\gamma_i, \delta_i\in\G$ such that $\be(\gamma_i)=\en(\delta_i)=x_i$, the units
$\en(\gamma_i)$ belong to $U$, and all the units $x_1, \ldots x_d,
\en(\gamma_1), \ldots, \en(\gamma_d), \be(\delta_1), \ldots,
\be(\delta_d)$ are pairwise different.

\begin{figure}
\centering
\includegraphics{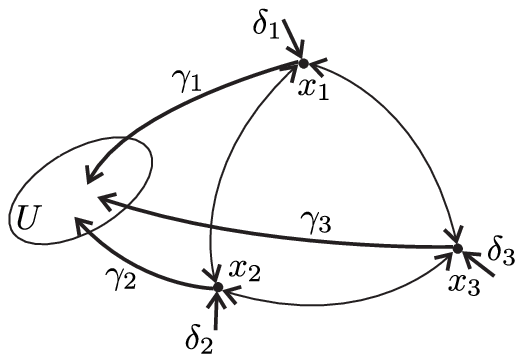}
\caption{Proof of simplicity of $\alt(\G)$}
\label{fig:simplicity}
\end{figure}

Since
$\G$ is \'etale, we can find  bisections $G_i\ni\gamma_i$,
$D_i\ni\delta_i$ and $F_{i, j}'\subset F_{i, j}$ such that
$\F'=\{F_{i, j}'\}_{i, j=1}^d$ is a multisection,
$F_{i, i}'=\be(G_i)=\en(D_i)$, $\en(G_i)\subset U$, and all the sets $\be(G_1),
\ldots \be(G_d), \en(G_1), \ldots, \en(G_d), \be(D_1), \ldots, \be(D_d)$ are pairwise disjoint.

Then $\mathcal{H}=\{G_jF_{i, j}'G_i^{-1}\}_{i, j=1}^d$ is a
multisection with domain in $U$. For every $i$ denote $C_i=G_i\cup
D_i^{-1}\cup D_iG_i^{-1}$. Then $\be(C_i)=\en(C_i)$, the sets
$\be(C_i)$, $i=1, \ldots, d$, are disjoint, and
$g=\bigcup_{i=1}^dC_i\cup\left(\G^{(0)}\setminus\bigcup_{i=1}^d\be(C_i)\right)$
is an element of $\alt(\G)$. It satisfies $G_jF_{i, j}'G_i^{-1}=gF_{i,
  j}g^{-1}$ and $g\alt(\F')g^{-1}=\alt(\mathcal{H})\le\alt(\G|_U)$. It
follows that $\alt(\F')\le g^{-1}\alt(\G|_U)g$.

We have shown that for every point $x_1\in F_{1, 1}$ there exists a
restriction $\F'$ of $\F$, and an element $g\in\alt(\G)$ such that
$x_1$ belongs to the domain of $\F'$, and $\alt(\F')\le
g^{-1}\alt(\G|_U)g$. It follows then from compactness of $F_{1, 1}$ and
Proposition~\ref{pr:covering} that $\alt(\F)$ is contained in the
normal closure of $\alt(\G|_U)$ in $\alt(\G)$.
\end{proof}

It was proved by H.~Matui (see~\cite[Theorem~4.7]{matui:fullonesided}
and~\cite[Theorem~4.16]{matui:fullonesided}) that if $\G$ is a minimal
groupoid of germs belonging
either to the class of \emph{almost finite} or to the class of
\emph{purely infinite} groupoids (see~\cite{matui:fullonesided} for
the definitions), then any non-trivial subgroup of $\full\G$
normalized by the commutator subgroup $\mathsf{D}(\G)$ contains
$\mathsf{D}(\G)$.  Theorem~\ref{th:simple} implies that we have
$\mathsf{D}(\G)=\alt(\G)$ for such groupoids. It would be interesting
to find an example of a minimal groupoid of germs $\G$ such that
$\alt(\G)\ne\mathsf{D}(\G)$.

\subsection{Example: AF-groupoids and LDA-groups}
\label{ss:AFLDA}

The groups $\alt(\G)$ for a special class of \emph{$AF$ groupoids}
are well known, though they are usually defined without using terminology of
groupoids.

A \emph{Bratteli diagram} $B$ consists of two sequences $V_1, V_2, \ldots$
and $E_1, E_2, \ldots$ of finite sets together with maps $\be_n:E_n\arr
V_n$ and $\en_n:E_n\arr V_{n+1}$. We interpret $e\in E_n$ as an edge
connecting the vertices $\be_n(e)\in V_n$ and $\en_n(e)\in V_{n+1}$. We
assume that the maps $\be_n$ and $\en_n$ are surjective.

The associated \emph{space of paths} $\mathcal{P}$ of the diagram $B$ is the space
of sequences $(e_1, e_2, \ldots)\in E_1\times E_2\times\cdots$ such
that $\en_n(e_n)=\be_{n+1}(e_{n+1})$ for every $n\ge 1$. The topology is
induced by the direct product topology on $E_1\times E_2\times\cdots$,
where $E_n$ are discrete.

A \emph{finite path} is a sequence $(e_1, e_2, \ldots, e_n)$ such that
$\en_i(e_i)=\be_{i+1}(e_{i+1})$ for all $i=1, 2, \ldots, n-1$. The
vertices $\be(e_1)$ and $\en_n(e_n)$ are called the \emph{beginning}
and the \emph{end} of the path, respectively.

Suppose that $(e_1, e_2, \ldots, e_n)$, $(f_1, f_2, \ldots, f_n)$
is a pair of finite paths with coinciding ends. This
pair defines a homeomorphisms between clopen subsets of $\mathcal{P}$
mapping every path of the form $(e_1, e_2, \ldots, e_n, e_{n+1},
e_{n+2}, \ldots)$ to the path $(f_1, f_2, \ldots, f_n, e_{n+1},
e_{n+2}, \ldots)$. Let $\mathfrak{B}$ be the groupoid of germs of the
semigroup generated by such locally homeomorphisms. Such groupoids,
i.e., groupoids associated in the described way with Bratteli
diagrams, are called \emph{AF groupoids}, see~\cite{matui:etale}.

For $v\in V_n$ denote by $d(v)$ the number of finite paths $(e_1, e_2,
\ldots, e_n)$ ending in $v$. It is easy to see that
$\alt(\mathfrak{B})$ is isomorphic to a direct limit of the direct
products $G_n=\prod_{v\in V_n}\alt_{d(v)}$. The embeddings $G_n\arr
G_{n+1}$ are block-diagonal. Namely, if $\alpha\in\alt_{d(v)}$ is a
permutation of the set of paths ending in $v$ then it acts on all
continuations of the paths only changing the beginnings by
permutation $\alpha$, so that $\alpha$ is copied $|\be_{n+1}^{-1}(v)|$
times in the direct product $G_{n+1}$.
See more details
in~\cite{lavnek:diagonal}.

The groupoid $\mathfrak{B}$ is minimal if and only if
for every two vertices $v_1, v_2\in \bigcup_{n\ge 1}V_n$ there
exist finite paths starting in $v_1$ and $v_2$ and ending in one
vertex. If this condition is satisfied, then the group
$\alt(\mathfrak{B})$ is simple. Such simple block-diagonal limits of
direct products of alternating groups form an important class of
simple locally finite groups, see~\cite{leinenpuglisi:confined,leinenpuglisi:1type}.

\section{Expansive groupoids and finite generation of $\alt(\G)$}

\subsection{Compactly generated and expansive groupoids}

\begin{defi}
\label{def:compactgenset}
A groupoid $\G$ is said to be \emph{compactly generated} if there
exists a compact set $S\subset\G$ (called a \emph{generating set}) such that $\G=\bigcup_{n\ge 1}(S\cup
S^{-1})^n$.
\end{defi}

Here $(S\cup S^{-1})^n$ is the set of products
$\gamma_1\gamma_2\cdots\gamma_n$ such that $\gamma_i\in S\cup S^{-1}$.

The general definition of a compactly generated groupoid (in the case
when $\G^{(0)}$ is not a Cantor set) is slightly different,
see~\cite{haefliger:compactgen,nek:hyperbolic}.

Suppose that $S$ is a compact generating set of $\G$. Then any set
$S_1\supseteq S$ is also a generating set. Consequently, there exists
a compact open generating set of $\G$.

We will need the following technical lemma later.

\begin{lemma}
\label{lem:nounits}
If $\G$ is compactly generated and every orbit has more than one point, then there exists a compact open
generating set $S$ of $\G$ such that for every $\gamma\in S$ we have $\be(\gamma)\ne\en(\gamma)$.
\end{lemma}

\begin{proof}
Let $S$ be an arbitrary compact open generating set of $\G$. Assume
that $\G^{(0)}\subset S$ (otherwise, replace $S$ by $\G^{(0)}\cup S$). 

Let $\mathcal{S}$ be a finite cover of $S$ by bisections. Suppose that $F\in\mathcal{S}$ and $g\in F$ are such that $\be(g)=\en(g)$.  There exists $h\in\G$ such that $\be(h)=\be(g)$ and $\en(h)\ne\be(h)$. We can choose a sufficiently small bisection $U$ containing $h$ such that $\be(U)\cap\en(U)=\emptyset$ and $\be(UF)\cap\en(UF)=\emptyset$. Then $U_g=U^{-1}UF$ is a neighborhood of $g$. Denote $W_g=U\cup UF$. Then $W_g$ does not contain elements of isotropy groups, and $U_g$ is contained in the groupoid generated by $W_g$.

 If $g\in F$ is such that $\be(g)\ne\en(g)$, then choose $U_g\subset F$ such that $\be(U_g)\cap\en(U_g)=\emptyset$, and denote $W_g=U_g$. The collection $\{U_g\}_{g\in F}$ is an open covering of $F$, hence there exists a finite sub-covering $\{U_{g_i}\}$. Let $F'$ be the union of the corresponding sets $W_g'$. Then $F'$ is compact, $F$ is contained in the groupoid generated by $F'$, and $F'$ does not contain elements of isotropy groups.
\end{proof}

\begin{defi}
Let $S$ be an open compact generating set of $\G$. A finite cover $\mathcal{S}$
of $S$ by  bisections is said to be \emph{expansive} if
the set
$\bigcup_{n\ge 1}\left(\mathcal{S}\cup\mathcal{S}^{-1}\right)^n$ is a
basis of topology of $\G$.
\end{defi}

\begin{proposition}
Suppose that there exists an expansive cover of a compact generating
set of  $\G$. Then for every compact generating set $S$
of $\G$ any finite cover $\mathcal{S}$ of $S$ by sufficiently small
 bisections is expansive.
\end{proposition}

\begin{proof}
Note that if $\mathcal{S}$ is an expansive set of  bisections,
then any finite set of  bisections containing
$\mathcal{S}$ is also expansive. If $\mathcal{S}_1$ is a
cover subordinate to an expansive cover, then
$\mathcal{S}_1$ is also expansive.

Let $S$ be a compact generating set such that there exists an
expansive cover $\mathcal{S}$ of $S$. Let $S_1$ be another compact
generating set. Let $\mathcal{S}_1$ be an arbitrary cover of $S_1$ by
bisections. For every $F\in\mathcal{S}$ and every $\gamma\in F$ there
exists a product of elements of $\mathcal{S}_1$ containing
$\gamma$. It follows that $F$ can be covered by a finite number of
products of elements of $\mathcal{S}_1$. Replacing $\mathcal{S}_1$ by
a refinement, we can find for every $F\in\mathcal{S}$ and every
$\gamma\in F$ a product $F_1$ of a finite number of elements of
$\mathcal{S}_1$ such that $\gamma\in F_1\subset F$. Then expansivity
of $\mathcal{S}_1$ will follow from expansivity of
$\mathcal{S}$. Since any cover of $S_1$ subordinate to $\mathcal{S}_1$
is also expansive, any cover of $S_1$ by sufficiently small bisections
is expansive, by Lebesgue's Covering Lemma.
\end{proof}

\begin{proposition}
\label{pr:bebasis}
Suppose that $\mathcal{S}$ is a finite cover by bisections of a compact generating
set $S$. Denote, for $x\in\G^{(0)}$ and $n\ge 1$, by $U_n(x)$
intersection of all sets of the form $\be(F)$
containing $x$, for $F\in(\mathcal{S}\cup\mathcal{S}^{-1})^k$ and $k\le n$. Then the following
conditions are equivalent.
\begin{enumerate}
\item The cover $\mathcal{S}$ is expansive.
\item The sets $\be(F)$ for $F\in\bigcup_{n\ge 1}(\mathcal{S}\cup\mathcal{S}^{-1})^n$ form a
basis of topology on $\G^{(0)}$.
\item For every two different units $x, y\in\G^{(0)}$ there exist
  $A_1, A_2\in\bigcup_{n\ge 1}(\mathcal{S}\cup\mathcal{S}^{-1})^n$
such that $x\in\be(A_1)$, $y\in\be(A_2)$, and $\be(A_1)\cap\be(A_2)=\emptyset$.
\item For every $x\in\G^{(0)}$ we have $\bigcap_{n\ge 1}U_n(x)=\{x\}$.
\end{enumerate}
\end{proposition}

\begin{proof}
The implication (1)$\Longrightarrow$(2) follows directly from the
definition of an \'etale groupoid.

For every bisection $F\subset\G$ we have $\be(F)=F^{-1}F$, hence
for every two products
$F_1F_2\cdots F_n$ and $G_1G_2\cdots G_m$ of elements of $\mathcal{S}\cup\mathcal{S}^{-1}$
the restriction \[G_1G_2\cdots G_m\be(F_1F_2\cdots F_n)\] of
$H_1H_2\cdots H_m$ to $\be(F_1F_2\cdots F_n)$ is equal to a product
of elements of $\mathcal{S}\cup\mathcal{S}^{-1}$. This shows that (2) implies (1).

It also shows that for any two products $F_1F_2\cdots F_n$ and $G_1G_2\cdots
G_m$ of elements of $\mathcal{S}\cup\mathcal{S}^{-1}$ the set $\be(F_1F_2\cdots
F_n)\cap\be(G_1G_2\cdots G_m)$ is a product of elements of
$\mathcal{S}\cup\mathcal{S}^{-1}$ (recall that if $U, V\subset\G^{(0)}$, then $U\cap
V=UV$).

The implication (2)$\Longrightarrow$(3) is obvious. Let us show the
converse implication. Let $x\in\G^{(0)}$, and let $U$ be an open
neighborhood of $x$. For every
unit $y\in\G^{(0)}\setminus U$ there exist products $U_y$ and $V_y$ of elements of
$\mathcal{S}\cup\mathcal{S}^{-1}$ such that $\be(U_y)\cap\be(V_y)=\emptyset$, $x\in\be(U_y)$,
and $y\in\be(V_y)$. By compactness of $\G^{(0)}\setminus U$, we can
find a finite collection $U_{y_1}, U_{y_2}, \ldots, U_{y_k}$,
$V_{y_1}, V_{y_2}, \ldots, U_{y_k}$ such that
$x\in\bigcap_{1\le i\le k}\be(U_{y_i})$,
$\be(U_{y_i})\cap\be(V_{y_i})=\emptyset$ for all $1\le i\le k$, and $\G^{(0)}\setminus
U\subset\bigcup_{1\le i\le k}\be(V_{y_i})$. By the above argument, the
intersection $\bigcap_{1\le i\le k}\be(U_{y_i})$ is a product of
elements of $\mathcal{S}\cup\mathcal{S}^{-1}$. It is an open neighborhood of $x$ contained
in $U$.

The implication (3)$\Longrightarrow$(4) is obvious. Suppose that (3) does not hold for
some points $x, y\in\G^{(0)}$. As we have seen above, for every $n\ge 1$ and
$x\in\G^{(0)}$, the set $U_n(x)$ is a product of elements of
$\mathcal{S}\cup\mathcal{S}^{-1}$. Then for every $n\ge 1$ the
intersection $U_n(x)\cap U_n(y)$ is non-empty. But $U_n(x)\cap
U_n(y)$, for $n\ge 1$, is a descending sequence of closed sets, hence
its intersection $\bigcap_{n\ge 1}U_n(x)\cap\bigcap_{n\ge 1}U_n(y)$ is
non-empty, so that (4) does not hold. This proves the implication (4)$\Longrightarrow$(3).
\end{proof}

\begin{proposition}
Let $U\subset\G^{(0)}$ be a clopen $\G$-transversal (i.e., a set intersecting every $\G$-orbit). Then the
groupoid $\G|_U$ is expansive if and only if $\G$ is expansive.
\end{proposition}

\begin{proof}
Suppose that $\G$ is expansive. Let $S$ be a compact generating set of
$\G$, and let $\mathcal{S}$ be a finite expansive cover of $S$ by
bisections. For every $x\in\G^{(0)}$ there exists a bisection $F$ such
that $x\in\be(F)$ and $\en(F)\subset U$. It follows that there exists
a finite set of bisections $\mathcal{T}$ such that the sets $\be(F)$,
for $F\in\mathcal{T}$, cover $\G^{(0)}$, and $\en(F)\subset U$ for
every $F\in\mathcal{T}$. We will assume that $U\in\mathcal{T}$.

Consider the set
$\mathcal{T}\mathcal{S}\mathcal{T}^{-1}$  of all possible products
$T_1FT_0^{-1}$, $T_i\in\mathcal{T}, F\in\mathcal{S}$. Note that
$H\subset\G|_U$ for every
$H\in\mathcal{T}\mathcal{S}\mathcal{T}^{-1}$.
For every $\gamma\in\G|_U$ there is a product $F=F_n\cdots F_2F_1$ of
elements of
$\mathcal{S}\cup\mathcal{S}^{-1}$ such that $\gamma\in F$, and the set
of such products form a basis of neighborhoods of $\gamma$. Find
elements $T_i\in\mathcal{T}$, $i=1, \ldots, n-1$,
such that $\en(F_iF_{i-1}\cdots F_1\be(\gamma))\in\be(T_i)$. Then
\[\gamma\in UF_nT_{n-1}^{-1}\cdot T_{n-1}F_{n-1}T_{n-2}^{-1}\cdots
T_1F_1U^{-1}\subset F.\]
It follows that $\mathcal{T}\mathcal{S}\mathcal{T}^{-1}$ is an
expansive cover of a compact generating set of $\G|_U$.

Suppose now that $\G|_U$ is expansive, and let $\mathcal{S}$ be an
expansive cover by bisections of a compact generating set of
$\G|_U$. Let $\mathcal{T}$ be as above. Then for every $\gamma\in\G$
there exist $T_0, T_1\in\mathcal{T}$ such that
$\be(\gamma)\in\be(T_i)$ and $\en(\gamma)\in\be(T_1)$. Then $T_1\gamma
T_0^{-1}\in\G|_U$. If $F_nF_{n-1}\cdots F_2F_1$ is a product of
elements of $\mathcal{S}\cup\mathcal{S}^{-1}$ containing $T_1\gamma
T_0^{-1}$, then $\gamma\in T_1^{-1}F_1F_{n-1}\cdots
F_2F_1T_0$. Similarly to the above, this implies that the set
$\mathcal{T}^{-1}\mathcal{S}\mathcal{T}$ is an expansive cover of a
compact generating set of $\G$.
\end{proof}

\begin{proposition}
\label{pr:expansivegroups}
Let $G$ be a finitely generated group acting by homeomorphisms on a
Cantor set $\mathcal{X}$. Let $G\times\mathcal{X}$ be the groupoid
of the action, and let $\G$ be the groupoid of germs of the action.
Then the following conditions are equivalent.
\begin{enumerate}
\item The groupoid $G\times\X$ is expansive.
\item The groupoid $\G$ is expansive.
\item The action of $G$ on $\mathcal{X}$ is expansive, i.e., there
  exists a neighborhood $W\subset\mathcal{X}\times\mathcal{X}$ of the
  diagonal such that if $x, y\in\X$ are such that $(g(x), g(y))\in W$
  for all $g\in G$, then $x=y$.
\item The dynamical system $(G, \mathcal{X})$ is a subshift, i.e.,
  there exists a finite set $X$ and a $G$-equivariant homeomorphism between  $\mathcal{X}$ and a $G$-invariant closed subset of $X^G$ (with the natural action of $G$ on $X^G$).
\end{enumerate}
\end{proposition}

\begin{proof}
Let us show that (1) or (2) imply (3).
If $A$ is a finite generating set of $G$, then $S\times\mathcal{X}$
and the set of germs of elements of $S$ are compact open generating
sets of the groupoids $G\times\X$ and $\G$, respectively. Suppose that
one of the groupoids $G\times\X$ or $\G$ is expansive. Then there exists an expansive cover
$\mathcal{S}$ of the above generating sets by bisections of the form
$(s, U)=\{(s, x)\;:\;x\in U\}$, where $s\in S$ and $U$ is a clopen subset of
$\X$. Moreover, we may assume that $U$ belongs to a finite
partition $\mathcal{P}$ of $\mathcal{X}$ into disjoint clopen
subsets. Let $W\subset\mathcal{X}\times\mathcal{X}$ be the set of
pairs $(x, y)$ such that $x$ and $y$ belong to the same element of
$\mathcal{P}$.

Let $x, y\in\mathcal{X}$ and $x\ne y$. Suppose, by contradiction, that
$(g(x), g(y))\in W$ for all $g\in G$. For every element
$A=(s, U)\in\mathcal{S}$ we have $\be(A)=U\in\mathcal{P}$. It
follows that $x\in\be(A)$ is equivalent to $y\in\be(A)$. The elements
$(s(x), s(y))$ again belong to one element of $\mathcal{P}$; hence, by
induction on $n$, we get that for every
sequence $A_1, A_2, \ldots, A_n$ of elements of $\mathcal{S}$ the
condition $x\in\be(A_1A_2\cdots A_n)$ is equivalent to the condition
$y\in\be(A_1A_2\cdots A_n)$. But it is a contradiction with
Proposition~\ref{pr:bebasis}.

Implication (3)$\Longrightarrow$(4) is classical. Let us repeat the
proof for completeness. Let $W$ be the neighborhood of the diagonal
satisfying the definition of expansivity of the action of $G$ on
$\mathcal{X}$. We can always take $W$ smaller, so we may assume that
there exists a finite partition $\mathcal{P}$ of $\mathcal{X}$ into
clopen sets such that $(x, y)\in W$ if and only if $x$ and $y$ belong
to the same element of $\mathcal{P}$. Consider now the following map
$\phi:\mathcal{X}\arr\mathcal{P}^G$:
\[\phi(\xi)(g)=U\Longleftrightarrow g(\xi)\in U,\]
for $\xi\in\mathcal{X}$, $g\in G$, $U\in\mathcal{P}$.

It is easy to see that the map $\phi$ is continuous and
$G$-equivariant. Since $\mathcal{X}$ and $\mathcal{P}^G$ are compact
Hausdorff, in order to prove that $\phi$ is a homeomorphism of
$\mathcal{X}$ with a closed subset of $\mathcal{P}^G$, it is enough to
show that $\phi$ is injective. But this follows directly from the
definition of expansivity.

Finally, let us prove that (4) implies (1) and (2). Let $\mathcal{X}$ be a
$G$-invariant closed subset of $X^G$, where $X$ is a finite
alphabet. Let $A$ be a finite symmetric generating set of $G$, and let
$S=\{(s, x)\;:\;s\in A, x\in\X\}$ be the corresponding generating set
of $G\times\X$ or $\G$. Consider the following cover of $S$:
\[\mathcal{S}=\{(s, C_x)\;:\;x\in X\},\]
where $C_x$ the cylindrical set of elements $w:G\arr X$ of
$\mathcal{X}$ such that $w(1)=x$.

Then $w\in\be((s_n, C_{x_n})(s_{n-1}, C_{x_{n-1}})\cdots (s_1, C_{x_1}))$ implies
that $w(1)=x_1$, $w(s_1)=x_2$, $w(s_2s_1)=x_3$, \ldots,
$w(s_{n-1}s_{n_2}\cdots s_1)=x_n$. Consequently, by taking a long enough product
$B_1B_2\cdots B_n$, we can specify arbitrarily large number of
coordinates of $w$ by the condition $w\in\be(B_1B_2\cdots B_n)$. This
implies that $G\times\X$ and $\G$ are expansive, by Proposition~\ref{pr:bebasis}.
\end{proof}

\begin{example}
\label{ex:nonexpansive}
Consider an arbitrary group
$G$ acting by automorphisms on a locally finite rooted tree
$T$. Consider the corresponding action by homeomorphisms on the
boundary $\partial T$ of the tree. Let $\G$ be the groupoid of the
action of $G$ on $\partial T$ (or the groupoid of germs of the
action). Then for every compact open subset $S\subset\G$ and for every finite
cover $\mathcal{S}$ of $S$ there exists
a level $L_n$ of the tree $T$ and a refinement $\mathcal{S}'$ of
$\mathcal{S}$ such that every element of $\mathcal{S}'$ is of the
form $\{g\}\times \partial T_v$, where
$T_v$ is the rooted subtree of $T$ with the root $v\in L_n$. The group
$G$ acts on $L_n$ by permutations, hence any product of elements of
$\mathcal{S}'$ is also of this form. It follows that $\G$ is not expansive.
\end{example}

\begin{example}
The AF-groupoids associated with Bratteli diagrams (see
Example~\ref{ss:AFLDA}) are also
never expansive. For every compact subset $C$ of such a groupoid there
exists $n$ such that every element of $C$ is a germ of a
transformation changing beginnings of paths of length at most $n$. The
set $\G_n$ of all such germs is equal to the union of a finite
number of multisections. It follows that for any finite set
$\mathcal{S}$ of open compact bisections the set of all products of
elements of $\mathcal{S}\cup\mathcal{S}^{-1}$ is finite.
\end{example}

\subsection{Finite generation of $\alt(\G)$}

\begin{theorem}
\label{th:finitelygenerated}
Suppose that $\G$ is expansive and every $\G$-orbit has at least 5 points.
Then the group $\alt(\G)$ is finitely generated.
\end{theorem}

Compare this result with~\cite[Proposition~4.4]{nek:gk}. Note that we do not require $\G$ to be a groupoid of germs.

\begin{proof}
If all $\G$-orbits have at least 5 points, then $\alt(\G)=\alt_5(\G)$, see Proposition~\ref{pr:anydegreegen}. Let us prove that $\alt_5(\G)$ is finitely generated.

Let $S$ be a compact open generating set of an expansive groupoid $\G$
such that $S=S^{-1}$ and $S$ does not contain elements $\gamma$ such
that $\be(\gamma)=\en(\gamma)$, see
Lemma~\ref{lem:nounits}. We may also assume that $|Sx|\ge 4$ for every
$x\in\G^{(0)}$, after possibly increasing $S$ by adding to it bisections
with disjoint source and target.

\begin{lemma} There exists a decomposition $\mathcal{P}=\{U_i\}$ of
  $\G^{(0)}$ into a finite disjoint
union of clopen sets such that every
$\G$-orbit intersects at least $5$ elements of $\mathcal{P}$.
\end{lemma}

\begin{proof}
Let $d$ be a metric on $\G^{(0)}$.
For every $x\in\G^{(0)}$ there exists $4$ elements $\gamma_1, \gamma_2, \gamma_3, \gamma_4\in\G$ such that $\be(\gamma_i)=x$, and the units
$\be(\gamma_1), \en(\gamma_1), \en(\gamma_2), \en(\gamma_3), \en(\gamma_4)$ are pairwise different. Hence, there exist
bisections $H_1, H_2, \ldots, H_4$ such that $\gamma_i\in H_i$,
$\be(H_1)=\be(H_2)=\ldots=\be(H_4)$, and the sets $\be(H_1),
\en(H_1), \ldots, \en(H_4)$ are pairwise disjoint. It follows that
there exists $\epsilon>0$ such that for every unit $y$ from the
$\epsilon$-neighborhood of $x$ the orbit of $y$ contains at least $5$
points (including $y$) that are on distance more than $\epsilon$ from
each other. It follows from the compactness of $\G^{(0)}$ 
that there exists $\epsilon>0$
such that for every $x\in\G^{(0)}$ there exist at least $5$ elements
of the orbit of $x$ that are on distance more than $\epsilon$ from
each other. Any partition $\mathcal{P}$ of $\G^{(0)}$ into clopen sets
of diameters less than $\epsilon$ will satisfy the conditions of
the lemma.
\end{proof}

Let $\mathcal{P}$ be a partition of $\G^{(0)}$ satisfying the conditions of the above lemma.
Since $S$ is compact and does not contain elements of isotropy groups,
there exists $\epsilon>0$ such that for every $\gamma\in S$ distance
between $\be(\gamma)$ and $\en(\gamma)$ is greater than
$\epsilon$. Consequently, replacing the partition $\mathcal{P}$ by a finite
refinement, we may assume that for every $\gamma\in S$ the units
$\be(\gamma)$ and $\en(\gamma)$ belong to different elements of
$\mathcal{P}$. We also may assume that for every $x\in\G^{(0)}$ there
exist elements $s_1, s_2, s_3, s_4\in S$ such that $\be(s_i)=x$ and
the points
$x, \en(s_1), \en(s_2), \en(s_3), \en(s_4)$ belong to pairwise
different elements of $\mathcal{P}$. (Recall that $|Sx|\ge 4$ for
every $x\in\G^{(0)}$, by the choice of $S$.)

Let $\mathcal{S}$ be an expansive finite cover of $S$ by
bisections such that $\mathcal{S}^{-1}=\mathcal{S}$ and
$\bigcup_{F\in\mathcal{S}}F=S$. Taking elements
of $\mathcal{S}$ small enough, we may assume that for every
$F\in\mathcal{S}$ the sets $\be(F)$ and $\en(F)$ are subsets of
different elements of $\mathcal{P}$.

Consider the set $T$ of elements $\gamma\in\bigcup_{k=1}^3S^k$ such that $\be(\gamma)$ and $\en(\gamma)$ belong to different elements of $\mathcal{P}$. Note that $S\subset T$ and $T$ is compact.

For every $\gamma\in T$ and every $H\in\bigcup_{k=1}^3\mathcal{S}^k$ such that $\gamma\in H$, we can find a degree $5$ multisection
$\F=\{F_{i, j}\}_{i, j=0}^4$ such that the sets $F_{i, i}$
are contained in pairwise different elements of $\mathcal{P}$, and
$\gamma\in F_{0, 1}\subset H$. By compactness,
we can find a finite set $M$ of multisections of this form such that the
corresponding sets $F_{0, 1}$ cover every element of $T$.

Let us denote by $\alt$ the group generated by the set
$\bigcup_{\F\in  M}\alt(\F)$. It is enough to prove that $\alt(\G)=\alt$.

\begin{lemma}
\label{lem:basissections} Suppose that $\gamma\in\G$ is such that $\be(\gamma)$ and $\en(\gamma)$ belong to different elements of $\mathcal{P}$. Then for every bisection $F\ni\gamma$ there
exists a degree $5$ multisection $\mathcal{H}=\{H_{i, j}\}$ such that $\alt(\mathcal{H})\le\alt$, $\gamma\in H_{0, 1}\subset F$, and the sets $H_{i, i}$ belong to pairwise different elements of $\mathcal{P}$.
\end{lemma}

\begin{proof}
It is enough to prove it for bisections $F$ of the form $G_1G_2\cdots
G_n$, where $G_i\in\mathcal{S}$. The lemma is true for $n\le 3$ by the
choice of $M$. Suppose that it is true for $n\ge 3$, let us show that it is
true for $F=G_0G_1\cdots G_n$ and $\gamma\in F$ such that $\be(\gamma)$ and
$\en(\gamma)$ belong to different elements of $\mathcal{P}$.

Denote by $\gamma_i\in G_i$ the elements such that
$\gamma=\gamma_0\gamma_1\cdots\gamma_n$. There exists $\gamma'\in S$
such that $\be(\gamma')=\be(\gamma_1)$ and the points $\en(\gamma'),
\be(\gamma), \en(\gamma)$ belong to pairwise different elements of $\mathcal{P}$.

We have $\gamma=(\gamma_0')\cdot(\gamma'\gamma_2\cdots\gamma_n)$ for
$\gamma_0'=\gamma_0\gamma_1(\gamma')^{-1}\in S^3$. Let
$G'\in\mathcal{S}$ be such that $\gamma'\in G'$, and denote
$G_0'=G_0G_1(G')^{-1}\in\mathcal{S}^3$. Then $F=G_0G_1\cdots G_n\supseteq
G_0'G'G_2\cdots G_n\ni\gamma$.

By the inductive hypothesis, there exists a degree $5$ bisection
$\mathcal{H}_0=\{H_{i, j}\}_{i, j=0}^4$, such that all the
components $H_{i, i}$  belong to pairwise different elements of $\mathcal{P}$, 
the element
$\gamma'\gamma_2\cdots\gamma_n$ belongs to $H_{0,
  1}\subset G'G_2\cdots G_n$, and
$\alt(\mathcal{H}_1)\subset\alt$. Assume that $H_{2, 2}$ and
$\en(\gamma)$ belong to different elements of $\mathcal{P}$, and let
$\mathcal{H}_1=\{H_{i, j}\}_{i, j=0}^2$ be the corresponding degree 3
multisection.

By the choice of $M$, there exists a multisection $\F_1=\{F_{i, j}\}_{i,
  j=0}^2$ which is a subset of an element of $M$, and such that
$\gamma_0'\in F_{0, 1}\subset G_0'$, and the elements of $\mathcal{P}$
containing $F_{1, 1}$ and $F_{2, 2}$ do not contain any component of
the domain of $\mathcal{H}_1$. Then, applying Proposition~\ref{pr:unionintersection} to
$\mathcal{H}_1$ and $\mathcal{F}_1$, we will find a degree $5$ multisection
$\mathcal{H}=\{T_{i, j}\}$ such that all components $T_{i, i}$ belong to
pairwise different elements of $\mathcal{P}$, and one of elements
$T_{i, j}\in\mathcal{H}$ satisfies $\gamma\in T_{i, j}\subset
G_0'G'G_2\cdots G_n\subset F$.
\end{proof}

\begin{lemma}
\label{lem:arbitraryfive}
Suppose that $\gamma\in\G$ is such that $\be(\gamma)\ne \en(\gamma)$, and $\{\be(\gamma), \en(\gamma)\}$ is contained in one element of $\mathcal{P}$. Then for every bisection $F\ni\gamma$ there exists a degree 5 multisection $\mathcal{H}=\{H_{i, j}\}_{i, j=0}^4$ such that  $\alt(\mathcal{H})\le\alt$, $\gamma\in H_{0, 1}\subset F$, and the sets $H_{0, 0}\cup H_{1, 1}, H_{2, 2}, H_{3, 3}, H_{4, 4}$ belong to pairwise different elements of $\mathcal{P}$.
\end{lemma}

\begin{proof} Let $U_0$ be the element of $\mathcal{P}$ containing $\be(\gamma)$ and $\en(\gamma)$. 
Let $\gamma'\in\G$ be such that $\be(\gamma')=\be(\gamma)$, and $\gamma'\notin U_0$. Then by Lemma~\ref{lem:basissections} there exists a degree 3 multisection $\mathcal{H}'=\{H_{i, j}'\}$  such that $\alt(\mathcal{H}')\le\alt$, $\gamma'\in H_{0, 1}'$, $H_{i, i}'$ are contained in pairwise different elements of $\mathcal{P}$. Then, again by Lemma~\ref{lem:basissections}, there exists a degree  $3$ multisection  $\mathcal{H}''=\{H_{i, j}''\}$ such that $\alt(\mathcal{H}'')\le\alt$,
the bisection $H_{0, 1}''$ is an arbitrarily small neighborhood of $\gamma'\gamma^{-1}$, the sets $H_{i, i}''$ are contained in pairwise different elements of $\mathcal{P}$, and the element of $\mathcal{P}$ containing $H_{2, 2}''$ is different from the element containing $H_{2, 2}'$. Then applying Proposition~\ref{pr:unionintersection} to $\mathcal{H}'$ and $\mathcal{H}''$ we obtain a degree 5 multisection satisfying the conditions of the lemma.
\end{proof}

Let now $\gamma_1, \gamma_2\in\G$ be such that $\be(\gamma_1)=\be(\gamma_2)$, and the points $\be(\gamma_1), \en(\gamma_1), \en(\gamma_2)$ are pairwise different. Let $\mathcal{H}_1$ be a multisection satisfying the conditions of Lemma~\ref{lem:basissections} or Lemma~\ref{lem:arbitraryfive} for $\gamma=\gamma_1$.

Note that the bisection $H_{0, 1}$ from $\mathcal{H}_1$ containing $\gamma_1$ can be made arbitrarily small, hence we may assume that its domain and range do not contain $\en(\gamma_2)$. There are at most three elements of $\mathcal{P}$ intersecting $\{\be(\gamma_1), \en(\gamma_1), \en(\gamma_2)\}$. Consequently, we can find a component $H_{i, i}$ of $\mathcal{H}_1$ different from $H_{0, 0}=\be(H_{0, 1})$ and $H_{1, 1}=\en(H_{0, 1})$ and not containing $\en(\gamma_2)$. Let $\mathcal{H}_1'\subset\mathcal{H}_1$ be the degree 3 sub-multisection with the components of the domain $H_{0, 0}, H_{1, 1}, H_{i, i}$. Then the domain of $\mathcal{H}_1'$ does not contain $\en(\gamma_2)$.

Let $\mathcal{H}_2$ multisection satisfying the conditions of Lemma~\ref{lem:basissections} or~\ref{lem:arbitraryfive} for $\gamma=\gamma_2$. Since the element of $\mathcal{H}_2$ containing $\gamma_2$ can be made arbitrarily small, we may assume that the component of the domain of $\mathcal{H}_2$ containing $\en(\gamma_2)$ does not intersect the domain of $\mathcal{H}_1'$. 

One of the components of $\mathcal{H}_2$ belongs to an element of $\mathcal{P}$ not containing $\en(\gamma_2)$ and not intersecting the domain of $\mathcal{H}_1$. It follows that
a subsets of $\mathcal{H}_2$ is a degree 3 multisection
$\mathcal{H}_2'$ such that the intersection of the domains of
$\mathcal{H}_1'$ and $\mathcal{H}_2'$ is contained in one component of
$\mathcal{H}_1'$ and one component of $\mathcal{H}_2'$. Then applying
Proposition~\ref{pr:unionintersection}, we get a degree 3 multisection
$\mathcal{H}$ such that two of its elements are arbitrarily small
neighborhoods of $\gamma_1$ and $\gamma_2$ and
$\alt(\mathcal{H})\le\alt$. Consequently, for every degree 3
multisection $\mathcal{F}=\{F_{i, j}\}_{i, j=0}^2$ there exists a
finite covering $F_{i, j}=\bigcup_{k=1}^mF_{i, j}^{(k)}$ such that
$\mathcal{F}^{(k)}=\{F_{i, j}^{(k)}\}_{i, j=0}^2$ are multisections
and $\alt(\mathcal{F}^{(k)})\le\alt$. 

It follows that the same statement is also true for degree 5
multisections (by Proposition~\ref{pr:unionintersection}). 
Proposition~\ref{pr:covering} implies then that $\alt_5(\G)=\alt$.
\end{proof}

\section{Cayley graphs of groupoids and expansivity}

Proposition~\ref{pr:expansivegroups} is usually a convenient criterion
for deciding if a groupoid coming from a group action is expansive. It
is harder to use it in the cases when an \'etale groupoid is defined
without reference to a group action. Here we analyze the expansivity
condition for general \'etale groupoids.

\subsection{General case}

Let $\G$ be an \'etale groupoid. We say that a set of
bisections $\mathcal{S}$ is a \emph{generating set} of if the union of its
elements is a generating set of $\G$ (according to
Definition~\ref{def:compactgenset}). The groupoid has a finite
generating set of bisections if and only if it is compactly generated.

If $S$ is a compact generating set of $\G$, then the \emph{Cayley
  graph} $\G(x, S)$ is the graph with the set of vertices
$\be^{-1}(x)$ in which there is an arrow from $\gamma_1$ to $\gamma_2$
if and only if there is $s\in S$ such that $\gamma_2=s\gamma_1$. 

For a given finite set of bisections
$\mathcal{S}$ generating $\G$, we consider the \emph{labeled Cayley
  graph} $\G(x, \mathcal{S})$ with the same set of vertices
$\be^{-1}(x)$ in which there is an arrow from
$\gamma_1$ to $\gamma_2$ labeled by an element $F\in\mathcal{S}$ if
and only if there exists $s\in F$ such that $\gamma_2=s\gamma_1$.

We denote by $\widehat\G(x, S)$ and $\widehat\G(x, \mathcal{S})$ the universal coverings of $\G(x, S)$ and $\G(x, \mathcal{S})$, respectively. The orientation and labeling of the universal coverings is lifted from the
orientation and labeling of $\G(x, S)$ and $\G(x, \mathcal{S})$.

Note that for every vertex $v$ of the graph $\G(x, \mathcal{S})$ and
for every label $F\in\mathcal{S}$ there exist at most one outgoing
arrow labeled by $F$ and at most one incoming arrow labeled by $F$.
It follows that any morphism $\phi:\G(x,
\mathcal{S})\arr\G(y, \mathcal{S})$ of oriented labeled graphs is
uniquely determined by the pair $(x, \phi(x))$. The same is true for
the (oriented labeled) universal coverings $\widehat\G(x, \mathcal{S})$
of the Cayley graphs.

The following is proved in~\cite[Corollary~2.3.4]{nek:hyperbolic}.

\begin{proposition}
Let $\G$ be a groupoid, and let $U\subset\G^{(0)}$ be a clopen
transversal. Let $S$ and $S_1$ be compact generating sets of $\G$ and
$\G|_U$, respectively. Then for every $x\in U$ the identical embedding
of the set of vertices of $\G|_U(x, S_1)$ into the set of vertices of
$\G(x, S)$ is a quasi-isometry of graphs.
\end{proposition}

A \emph{rooted graph} is a graph with a marked vertex. A morphism of
rooted graphs must map the root to the root. Every Cayley graph $\G(x,
\mathcal{S})$ is considered to be an oriented labeled rooted graph
with the root $x$.

If $(\Gamma, v)$ is a rooted connected
graph, then its rooted universal covering is the usual universal covering of
$\Gamma$ in which the root is any preimage of $v$ under the covering
map. Note that it is unique up to an isomorphism of rooted graphs. The
universal covering can be defined as the space of homotopy classes of
paths in $\Gamma$ starting at $v$, where the root is the trivial path
at $v$.
If $\Gamma$ is labeled and oriented, then the universal covering is also
labeled and oriented in the natural way.

\begin{proposition}
\label{pr:morphism}
Let $\mathcal{S}$ be a finite generating set of bisections. Then
$\mathcal{S}$ is non-expansive if and only if there
exist $x, y\in\G^{(0)}$ such that $x\ne y$ and there exists
a morphisms of rooted labeled graphs $\widehat\G(x,
\mathcal{S})\arr\widehat\G(y, \mathcal{S})$.
\end{proposition}

\begin{proof}
Every rooted path in $\G(x, \mathcal{S})$ corresponds to a product of
the elements of $\mathcal{S}$ and their inverses (one just has to read
the labels). Such a path exists if and only if $x$ belongs to the set of
sources of the corresponding product. It follows that the universal
covering is precisely the tree of all possible products $F$ of the
elements of $\mathcal{S}\cup\mathcal{S}^{-1}$ such that
$x\in\be(F)$. The proof is then finished by using condition (4) of
Proposition~\ref{pr:bebasis}.
\end{proof}

\begin{theorem}
\label{th:universalcov}
A compactly generated groupoid $\G$ is expansive if and only if there
exists a finite generating set $\mathcal{S}$ of bisections such that
for every pair of different points $x, y\in\G^{(0)}$ the universal
coverings $\widehat\G(x, \mathcal{S})$ and $\widehat\G(y, \mathcal{S})$
are non-isomorphic as rooted oriented labeled graphs.
\end{theorem}

\begin{proof}
Every expansive generating set $\mathcal{S}$ satisfies
Proposition~\ref{pr:morphism}, hence has pairwise non-isomorphic
rooted graphs $\widehat\G(x, \mathcal{S})$.

Suppose that $\mathcal{S}$ is a finite generating set of bisections
such that the rooted graphs $\widehat\G(x, \mathcal{S})$ are pairwise
non-isomorphic. Let $U_n(x)$ be the sets defined in
Proposition~\ref{pr:bebasis}. The number of possible sets of the form
$U_1(x)$ is finite. Therefore, there exists a finite generating set
$\mathcal{S}_1$ such that each element of $\mathcal{S}$ is a union of
elements of $\mathcal{S}_1$, and for every
$F\in\mathcal{S}_1$ and $x\in\be(F)$ we have
$U_1(x)\subset\be(F)$. The Cayley graph $\G(x, \mathcal{S})$ is
determined by the Cayley graph $\G(x, \mathcal{S}_1)$ and is obtained
from it just by relabeling of the edges.

By the definition of $\mathcal{S}_1$,
the label of every edge starting at a vertex $v\in\G(x,
\mathcal{S}_1)$ uniquely determines the labels of the edges adjacent
to $v$ in $\G(x, \mathcal{S})$. In particular, the number of vertices
adjacent to $v$ in $\G(x, \mathcal{S}_1)$ (which is the same as in $\G(x,
\mathcal{S})$) is uniquely determined by the label of any edge
starting at $v$.

This implies that every morphism
$\widehat\G(x, \mathcal{S}_1)\arr\widehat\G(y, \mathcal{S}_1)$ is an isomorphism
that induces an isomorphism $\widehat\G(x, \mathcal{S})\arr\widehat\G(y,
\mathcal{S})$. It follows, by Proposition~\ref{pr:morphism}, that
$\mathcal{S}_1$ is expansive.
\end{proof}

\subsection{Compactly presented groupoids}
It is in many cases easier to prove that the Cayley graphs $\G(x,
\mathcal{S})$ based at different points $x$ are different, than to
prove that their universal coverings are different.

In general, it is possible that all Cayley graphs $\G(x, \mathcal{S})$
are pairwise non-isomorphic (for different $x$), but the groupoid is
not expansive. For example, there are examples of finitely generated
groups acting by automorphisms of a rooted tree such that the Cayley
graphs of the associated groupoid of germs of the action on the
boundary of the tree are pairwise non-isomorphic (see,
for instance~\cite{angelidonnomatternagnibeda}).
But such groupoids can not be expansive, see~\ref{ex:nonexpansive}.

Note, however,
that we have the following reformulation of
Proposition~\ref{pr:expansivegroups}.

\begin{proposition}
Let $G$ be a finitely generated group acting on a Cantor set $\X$. Let
$\G=G\times\X$ be the groupoid of the action. Then $\G$ is expansive
if and only if there exists a finite generating set $\mathcal{S}$ such
that for every $x, y\in\X$, $x\ne y$, the Cayley graphs $\G(x,
\mathcal{S})$ and $\G(y, \mathcal{S})$ are non-isomorphic as rooted
labeled directed graphs.
\end{proposition}

Another class of groupoids for which it is enough to consider the
Cayley graphs to prove expansivity are \emph{compactly presented} groupoids.

For a graph $\Gamma$ and positive integer $R$, the \emph{Rips
  complex} $\Delta_R\Gamma$
is the simplicial complex with the set of vertices equal to
the set of vertices of $\Gamma$ in which a subset $A$ of the set of
vertices is a simplex if and only if its diameter is not more than
$R$.

\begin{defi}
Let $\G$ be an \'etale groupoid, and let $S$ be a compact open
generating set. The groupoid $\G$ is said to be
\emph{compactly presented} if there exists $R>0$ such that every Rips
complex $\Delta_R\G(x, S)$ is simply connected for every $x\in\G^{(0)}$.
\end{defi}

It is not hard to show that the definition does not depend on the
choice of the generating set.

\begin{theorem}
\label{th:compactlypresented}
Suppose that $\G$ is a compactly presented Hausdorff \'etale groupoid. It is
expansive if and only if there exists a finite generating set
$\mathcal{S}$ of bisections such that the graphs $\G(x,
\mathcal{S})$ and $\G(y, \mathcal{S})$ are non-isomorphic as
rooted oriented labeled graphs for every pair $x\ne y$.
\end{theorem}

\begin{proof}
If $\G$ is expansive, then a generating set $\mathcal{S}$ satisfying
the conditions of the theorem exists by Theorem~\ref{th:universalcov}.

Let us prove the converse statement. Let $\mathcal{S}$ be a generating
set satisfying the conditions of the theorem.

If $\G$ is a Hausdorff \'etale groupoid, then for every $\gamma\in\G\setminus\G^{(0)}$
there exists a bisection $U$ such that $\gamma\in U$ and
$U\cap\G^{(0)}=\emptyset$. Otherwise, $\gamma$ and $\be(\gamma)$ do
not have disjoint neighborhoods. In other words, $\G^{(0)}$ is
closed. Note that it is open by definition of an \'etale groupoid.

Denote by $N_R(x, \mathcal{S})$,
for $x\in\G^{(0)}$ and $R\ge 0$, the set of points $y\in\G^{(0)}$ such
that the balls of radius $R$ with centers $x$ and $y$ in the Cayley
graphs $\G(x, \mathcal{S})$ and $\G(y, \mathcal{S})$ are isomorphic as
rooted (with roots equal to the centers $x$ and $y$) labeled oriented
graphs. By the above, the sets $N_R(x, \mathcal{S})$ are clopen.

Replacing $\mathcal{S}$ by $(\mathcal{S}\cup\mathcal{S}^{-1})^n$ for
some $n$, if necessary, we may assume that the Rips complexes
$\Delta_1\G(x, \mathcal{S})$ are simply connected. We may also assume
that the sets $\be(F)$, for $F\in\mathcal{S}$, cover $\G^{(0)}$, so
that every vertex of every Cayley graph has at least one outgoing
edge.

Let $\mathcal{S}_1$
be a finite set of bisections such that every element of $\mathcal{S}$ is a
disjoint union of elements of $\mathcal{S}_1$, and for every
$F\in\mathcal{S}_1$ and $x\in\be(F)$ we have $\be(F)\subset N_1(x,
\mathcal{S})$. Then the Cayley graphs $\G(x, \mathcal{S})$ are
obtained from $\G(x, \mathcal{S}_1)$ by relabeling edges, and the
label of any edge $e$ of $\G(x, \mathcal{S}_1)$ uniquely determines the
1-ball in $\G(x, \mathcal{S})$ of the source of $e$. It also follows that if
$x\ne y$, then the rooted graphs $\G(x, \mathcal{S}_1)$ and $\G(y,
\mathcal{S}_1)$ are non-isomorphic.

Suppose that $\phi:\widehat\G(x, \mathcal{S}_1)\arr\widehat\G(y, \mathcal{S}_1)$ is a
morphism of rooted graphs. Note that by the same argument as in the
proof of Theorem~\ref{th:universalcov}, $\phi$ is an isomorphism.

The morphism $\phi$ induces a covering map $\phi':\widehat\G(x,
\mathcal{S}_1)\arr\G(y, \mathcal{S}_1)$, obtained by composing $\phi$
with the universal covering map $\widehat\G(y, \mathcal{S}_1)\arr\G(y,
\mathcal{S}_1)$.

Let us introduce an equivalence relation on the sets of vertices and
edges of $\widehat\G(x, \mathcal{S}_1)$ by the rule that $a_1\sim a_2$
if and only if $\phi'(a_1)=\phi'(a_2)$ and the images of $a_1$ and
$a_2$ in $\G(x, \mathcal{S}_1)$ coincide. Let $\Gamma$ be the quotient
of $\tilde\G(x, \mathcal{S}_1)$ by this equivalence relation. We have
natural maps $\psi:\Gamma\arr\G(y, \mathcal{S}_1)$ and
$\pi:\Gamma\arr\G(x, \mathcal{S}_1)$ induced by $\phi'$ and the
universal covering map, respectively.

Let $\Sigma=\{v_1, v_2, v_3\}$ be a face of
$\Delta_1\G(x, \mathcal{S}_1)$. Let $e$ be an edge in $\Sigma$, and
let $v_1=\be(e)$ and $v_2=\en(e)$. Then the label of $e$ in $\G(x,
\mathcal{S}_1)$ uniquely determines the labels and orientation of the
edges of $\Sigma$ in $\G(x, \mathcal{S})$. Moreover, any edge $f$ of
$\G(x, \mathcal{S}_1)$ (or $\G(y, \mathcal{S}_1)$) with the same label
as $e$ belongs in $\G(x, \mathcal{S})$ (resp.\ $\G(y, \mathcal{S})$)
to a face with the same labels of sides as in $\Sigma$. It follows
that the maps $\psi:\Gamma\arr\G(y, \mathcal{S}_1)$ and
$\pi:\Gamma\arr\G(x, \mathcal{S}_1)$ induce coverings of the
respective Rips complexes. But $\Delta_1\G(x, \mathcal{S}_1)$ and
$\Delta_1\G(y, \mathcal{S}_1)$ are simply connected, hence $\psi$ and
$\pi$ are isomorphisms.
\end{proof}

\subsection{Quasicrystals}

Consider the following class of groupoids,
as an example of an application of Theorem~\ref{th:compactlypresented}.

\begin{defi}
\label{def:quasicrystal}
We say that a subset $Q$ of the vector space $\R^n$ is  a
\emph{quasicrystal} if it satisfies the following conditions.
\begin{enumerate}
\item It is a \emph{Delaunay set}:
\begin{itemize}
\item there exists $R>0$ such that for every $x\in\R^n$ there exists
  $q\in Q$ such that $\|x-q\|<R$;
\item there exists $\delta>0$ such that for any $q_1, q_2\in Q$, $q_1\ne q_2$,
  we have $\|q_1-q_2\|>\delta$.
\end{itemize}
\item It has \emph{finite local complexity}:
for every $R>0$ the set $\mathcal{B}_R(Q)$ of sets of the form
  $B_R(q)\cap Q-q$ for $q\in Q$ is finite.
\item It is \emph{repetitive}: for every $R>0$ and $q\in Q$ there exists $D>0$ such that for
  every $x\in\R^n$ there exists $p\in Q$ such that $\|x-p\|<D$ and
  $B_R(p)\cap Q-p=B_R(q)\cap Q-q$.
\end{enumerate}
\end{defi}

Here $B_R(x)=\{y\in\R^n\;:\;\|x-y\|<R\}$
is the open ball of radius $R$ with center in $x$.

We say that a quasicrystal $Q\subset\R^n$ is \emph{aperiodic} if
$Q+v=Q$ for $v\in\R^n$ implies $v=\vec 0$.

Many examples of quasicrystals can be constructed using the
\emph{cut-and-project} method, see~\cite{bruijn:pen1}.

\begin{defi}
The \emph{hull}, see~\cite{bellisardherrmannzarrouati:hulls},
$\mathcal{Q}$ of a quasicrystal $Q$ is the set of all
subsets $Q'\subset\R^n$ such that for any $R>0$ there exists $q\in Q$
such that $B_R(\vec 0)\cap Q'=B_R(q)\cap Q-q$.
\end{defi}

One can show that every element of the hull is a quasicrystal.
Note that the zero vector $\vec 0$ belongs to every element of the hull.

Let us introduce a metric on $\mathcal{Q}$ setting $d(Q_1, Q_2)$ equal
to $R^{-1}$, where $R$ is supremum of radii $r$ such that $B_r(\vec
0)\cap Q_1=B_r(\vec 0)\cap Q_2$. It is easy to see that $\mathcal{Q}$
is compact and totally disconnected.

If $Q$ is aperiodic, then the space $\mathcal{Q}$ has no isolated
points, and hence is homeomorphic to the Cantor set.

Groupoids and inverse semigroups are appropriate tools for the study
of symmetries of quasyperiodic structures,
see~\cite{bellissardjuliensavinien,kellendonk_putnam,kellendonk_lawson}.
Let us describe a natural \'etale groupoid associated with a quasicrystal.
Fix an aperiodic quasicrystal $Q$.
Denote by $\mathfrak{Q}$ the following groupoid. Its elements are
pairs $(P, q)$, where $P\in\mathcal{Q}$ and $q\in P$. They are
multiplied by the rule
\[(P_1, q_1)(P_2, q_2)=(P_2, q_1+q_2),\]
where the product is defined if and only if $P_2-q_2=P_1$. The set of
units is equal to the set of elements of the form $(P, \vec 0)$, hence it
is naturally identified with the hull $\mathcal{Q}$. We have then $\be(P, q)=P$
and $\en(P, q)=P-q$. In other
words, the element $(P, q)$ of $\mathfrak{Q}$ corresponds to the
change of the origin in $P$ from $\vec 0$ to $q$. For $P\in\mathcal{Q}$, the $\mathfrak{Q}$-orbit of $P$ is the set of quasicrystals of the
form $P-p$ for $p\in P$.

The natural topology on $\mathfrak{Q}$ is defined in the same way as
for $\mathcal{Q}$. The distance from $(P_1, q_1)$ to $(P_2, q_2)$ is
equal to 1 if $q_1\ne q_2$, and to $R^{-1}$, where $R$ is the supremum
of radii $r$ such that $B_r(\vec 0)\cap P_1= B_r(\vec 0)\cap P_2$, if
$q_1=q_2$. It is easy to show that $\mathfrak{Q}$ is \'etale,
Hausdorff, principal, and minimal.

Denote by $\mathcal{B}_R$ the set of sets of the form $B_R(q)\cap Q-q$.

\begin{defi}
 We say that a permutation $\alpha:Q\arr Q$
is \emph{defined by local rules} if there exist $R>0$ and a map
$W:\mathcal{B}_R\arr\R^n$ such that for every $q\in Q$ we have
\[\alpha(q)-q=W((B_R(q)\cap Q)-q).\]
In other words, if $\alpha(q)-q$ is uniquely determined by the
$R$-neighborhood of $q$ in $Q$.
\end{defi}

The following is straightforward.

\begin{proposition}
The group $\full Q$ is isomorphic to the topological full
group $\full{\mathfrak{Q}}$. Moreover, the action of $\full Q$
coincides with the action of $\full{\mathfrak{Q}}$ on the
$\mathfrak{Q}$-orbit of $Q\in\mathcal{Q}$.
\end{proposition}

Let us denote by $\alt(Q)$ the group $\alt(\mathfrak{Q})$. It can be
defined in terms of $Q$ as the subgroup of $\full Q$ generated by
permutations $\alpha$ such that the orbits of $\langle\alpha\rangle$ are
of lengths 1 and 3 only.

\begin{theorem}
\label{th:quasicrystals}
Let $Q$ be a quasicrystal. Then $\alt(Q)$ is a simple finitely
generated group.
\end{theorem}

\begin{proof}
For $P\in\mathcal{Q}$, denote by $\Delta_RP$ the Rips complex of $P$,
i.e., the simplicial complex with the set of vertices $P$ in which a
set is a simplex if and only if its diameter is less than or equal to
$R$. It follows from the definition of a quasicrystal that there
exists $R$ such that $\Delta_RP$ is quasi-isometric to $\R^n$ and
simply connected for every $P\in\mathcal{Q}$.

Denote by $T_v$, for $v\in\R^n$, the set of elements
of $\mathfrak{Q}$ of the form $(P, v)$ for $P\in\mathcal{Q}$. The set
$T_v$ is a (possibly empty) $\mathfrak{Q}$-bisection. Let
$\mathcal{S}$ be the set of all non-empty bisections $T_v$ for
$\|v\|\le R$. It is finite and satisfies the conditions of
Theorem~\ref{th:compactlypresented}, since the
complexes $\Delta_RP$ are simply connected for all $P\in\mathcal{Q}$
\end{proof}

\section{The quotient $\symm(\G)/\alt(\G)$}

The following proposition follows directly from the definition of the
homology groups $H_0(\G, A)$.

\begin{proposition}
The group $H_0(\G, \Z/2\Z)$ is isomorphic to the abelian group given by the following
presentation.
\begin{itemize}
\item Its generators $1_U$ are labeled by clopen subsets
$U\subset\G^{(0)}$.
\item For every clopen set $U$ we have $1_U+1_U=0$.
\item For every decomposition $U=U_1\sqcup
U_2\sqcup\cdots\sqcup U_n$ of a clopen set into a disjoint union of clopen
subsets, we have  $1_U=1_{U_1}+1_{U_2}+\cdots+1_{U_n}$.
\item For
every bisection $F\subset\G$, we have $1_{\be(F)}=1_{\en(F)}$.
\end{itemize}
\end{proposition}

Denote $\tau_F=F\cup F^{-1}\cup
(\G^{(0)}\setminus(\be(F)\cup\en(F))$, where $F$ is a bisection such
that $\be(F)\cap\en(F)=\emptyset$. 

\begin{theorem}
\label{th:H0} Suppose that every $\G$-orbit has at least 3 elements.
Then the correspondence $1_{\be(F)}\mapsto \tau_F$ induces a well defined epimorphism
from $H_0(\G, \Z/2\Z)$ to $\symm(\G)/\alt(\G)$. 
\end{theorem}


\begin{proof}
Note that for every $x\in\G^{(0)}$ and for every sufficiently small clopen
neighborhood $U\subset\G^{(0)}$ of $x$ there exists a bisection $F$
such that $\be(F)=U$ and $\en(F)\cap\be(F)=\emptyset$. It follows that elements of
the form $1_{\be(F)}$ for bisections $F\subset\G$ such that
$\be(F)\cap\en(F)=\emptyset$ generate $H_0(\G, \Z/2\Z)$.

Let us denote by $t_F$ the image of $\tau_F$ in $\symm(\G)/\alt(\G)$.
We have to show that the generators $t_F$ of $\symm(\G)/\alt(\G)$
depend only on $\be(F)$, and satisfy the defining relations of
$H_0(\G, \Z/2\Z)$.

The generators $t_F$ of $\symm(\G)/\alt(\G)$ are obviously of order 2,
and if $F=F_1\sqcup F_2\sqcup\cdots\sqcup F_n$ is a decomposition of $F$ into a disjoint union of  bisections, then $t_F=t_{F_1}t_{F_2}\cdots t_{F_n}$.

\begin{lemma}
\label{lem:F1F2easy}
If bisections $F_1, F_2$ are such that $\be(F_1)=\be(F_2)$, and the sets
$\be(F_1)$, $\en(F_1)$, $\en(F_2)$, and $\be(F_1)$ are pairwise disjoint, then $t_{F_1}=t_{F_2}$.
\end{lemma}

\begin{proof}
The element $\tau_{F_1}\tau_{F_2}^{-1}$ belongs to $\alt(\G)$.
\end{proof}

\begin{lemma}
\label{lem:F1F2}
The element $t_F$ depends only on $\be(F)$.
\end{lemma}

\begin{proof}
Suppose that bisections $F_1$ and $F_2$ are such that
$\be(F_1)=\be(F_2)$ and
$\be(F_1)\cap\en(F_1)=\be(F_2)\cap\en(F_2)=\emptyset$. Let
$x\in\be(F_i)$. If $\en(F_1x)\ne\en(F_2x)$, then for all sufficiently small
clopen neighborhoods $U$ of $x$ the bisections $F_1U$ and $F_2U$ will
satisfy the conditions of Lemma~\ref{lem:F1F2easy}, hence
$t_{F_1U}=t_{F_2U}$. Suppose that $\en(F_1x)=\en(F_2x)$. Then there
exists $y$ in the orbit of $x$ different from $x$ and from
$\en(F_ix)$. Let $H$ be a bisection such that $x\in\be(H)$ and
$y\in\en(H)$. Then for all sufficiently small clopen neighborhoods $U$
of $x$, by Lemma~\ref{lem:F1F2easy}, we have $t_{HU}=t_{F_1U}$ and
$t_{HU}=t_{F_2U}$, which implies that $t_{F_1U}=t_{F_2U}$.

We have proved that for every $x\in\be(F_i)$ and all sufficiently
small clopen neighborhoods $U$ of $x$ we have $t_{F_1U}=t_{F_2U}$. It
follows that we can decompose $\be(F_i)$ into a finite disjoint union
$\be(F_i)=U_1\sqcup U_2\sqcup\cdots\sqcup U_n$ of clopen sets such
that $t_{F_1U_i}=t_{F_2U_i}$ for all $i$. But then
$t_{F_1}=t_{F_1U_1}t_{F_1U_2}\cdots
t_{F_1U_n}=t_{F_2U_1}t_{F_2U_2}\cdots t_{F_2U_n}=t_{F_2}$.
\end{proof}

It remains to show that the elements $t_F$ commute. In view of
Lemma~\ref{lem:F1F2}, by a slight abuse of notation we will
denote by $t_U$ any element $t_F$ such that $U=\be(F)$. The
element $t_U$ is defined for all sufficiently small clopen sets $U$.

Let $U_1, U_2\subset\G^{(0)}$ be clopen sets such that $t_{U_1}$ and
$t_{U_2}$ are defined. Then $t_{U_1}=t_{U_1\setminus U_2}t_{U_1\cap
  U_2}$ and $t_{U_2}=t_{U_2\setminus U_1}t_{U_1\cap U_2}$. It follows
that it is sufficient to prove that $t_{U_1}$ and $t_{U_2}$ commute if
$U_1$ and $U_2$ are disjoint (since then we will conclude that
$t_{U_1\setminus U_2}, t_{U_2\setminus U_1}$, and $t_{U_1\cap U_2}$
pairwise commute).

Suppose that $U_1$ and $U_2$ are disjoint. Let $x_i\in U_i$. 
If $x_1$ and $x_2$ belong to one $\G$-orbit, then there exists a bisection $F$ such that $x_1\in\be(F)$, $x_2\in\en(F)$, and $\be(F)\cap\en(F)=\emptyset$. Then for any clopen sets $V_1\subset\be(F)$ and $V_2\subset\en(F)$ the elements $t_{FV_1}$ and $t_{V_2F}$ commute. But $t_{FV_1}=t_{V_1}$ and $t_{V_2F}=t_{V_2}$ by Lemma~\ref{lem:F1F2}. We have shown that there exist neighborhoods $W_1=\be(F)$, $W_2=\en(F)$ of $x_1$ and $x_2$, respectively, such that for all clopen sets $V_i\subset W_i$ the elements $t_{V_1}$ and $t_{V_2}$ commute.

Suppose now that $x_1$ and $x_2$ belong to different $\G$-orbits. Then there exist elements
$\gamma_i\in\G$ such that $\be(\gamma_i)=x_i$, and points $x_1, x_2,
\en(\gamma_1), \en(\gamma_2)$ are pairwise different. It follows that
there exist bisections $F_1, F_2$ such that $\gamma_i\in F_i$,
and the sets $\be(F_1), \en(F_1), \be(F_2), \en(F_2)$ are pairwise
disjoint. Then for any clopen subsets
$V_i\in\be(F_i)$, the elements $\tau_{F_1V_1}$ and $\tau_{F_2V_2}$
commute in $\symm(\G)$, hence $t_{V_1}$ and $t_{V_2}$ commute in
$\symm(\G)/\alt(\G)$. 
We have shown again that there exist neighborhoods $W_i=\be(F_i)$ of $x_i$ such that for any clopen subsets $V_i\subset  W_i$ the elements $t_{V_1}$ and $t_{V_2}$ commute.

It follows then from the compactness of $U_1$ and $U_2$ that we can decompose $U_1$ and $U_2$ into finite
disjoint unions of clopen sets $U_1=A_1\sqcup A_2\sqcup\ldots\sqcup A_n$ and
$U_2=B_1\sqcup B_2\sqcup\ldots\sqcup B_m$ such that $t_{A_i}$ and
$t_{B_j}$ commute for all $i$ and $j$. But this implies that
$t_{U_1}=t_{A_1}\cdot t_{A_2}\cdots t_{A_n}$ and
$t_{U_2}=t_{B_1}t_{B_2}\cdots t_{B_m}$ commute.
\end{proof}

It was shown by H.~Matui (see~\cite[Theorem~7.5]{matui:etale}
and~\cite[Theorem~7.13]{matui:etale}) that for almost finite groupoids
of germs the index map $I:\full{\G}\arr H_1(\G, \Z)$ is surjective,
and its kernel $\mathsf{K}(\G)$ is equal to $\symm(\G)$. Recall that if $\G$ is in addition
minimal, then $\alt(\G)=\mathsf{D}(\G)$. For groupoids in this class
we get from Theorem~\ref{th:H0} the following description of the abelianization of the
topological full group.

\begin{proposition}
Let $\G$ be an almost finite minimal groupoid of germs. Then the
abelianization $\full\G/\mathsf{D}(\G)$ is an extension of a quotient
of $H_0(\G, \Z/2\Z)$ by $H_1(\G, \Z)$, i.e., there exists an exact
sequence
\[H_0(\G, \Z/2\Z)\arr\full\G/\mathsf{D}(\G)\arr H_1(\G, \Z)\arr 0.\]
If $\G$ is expansive, and the groups $H_0(\G, \Z/2\Z)$ and $H_1(\G,
\Z)$ are finitely generated, then $\full\G$ is finitely generated.
\end{proposition}

See more on abelianization of the full group in~\cite{matui:product}.

\begin{example}
The epimorphism $H_0(\G, \Z/2\Z)\arr\symm(\G)/\alt(\G)$ is not
necessarily one-to-one. Consider the groupoid $\G$ with the space of
units $\G^{(0)}$ equal to the space $\{1, 2, 3\}^{\mathbb{N}}$ of
infinite sequences $x_1x_2\ldots$ of elements of the alphabet $X=\{1,
2, 3\}$. Consider the inverse semigroup generated by the
transformations $T_x:x_1x_2\ldots\mapsto xx_1x_2\ldots$, for $x\in X$,
and the transformation $a$ given by the recursive rules:
\[a(1x_2x_3\ldots)=2x_2x_3\ldots,\quad
a(2x_2x_3\ldots)=1a(x_2x_3\ldots),\quad a(3x_2x_3\ldots)=3x_2x_3\ldots.\]
Let $\G$ be the groupoid of germs of this semigroup.

The full group $\full\G$ consists of transformations of the following
type. Let $v_1, v_2, \ldots, v_n$ and $u_1, u_2, \ldots, u_n$ be two
sequences of finite words such that $X^{\mathbb{N}}$ is equal to the
disjoint unions $v_1X^{\mathbb{N}}\cup v_2X^{\mathbb{N}}\cup\cdots\cup
v_nX^{\mathbb{N}}$ and $u_1X^{\mathbb{N}}\cup u_2X^{\mathbb{N}}\cup\cdots\cup
u_nX^{\mathbb{N}}$. Let $a^{m_1}, a^{m_2}, \ldots, a^{m_n}$ be arbitrary
elements of the infinite cyclic group generated by $a$.
Then the corresponding element of $\full\G$ is the (set of
germs of the) homeomorphism $v_iw\mapsto u_ia^{m_i}(w)$. We will
denote such an element by the table $\left(\begin{array}{cccc}v_1 &
    v_2 & \ldots & v_n\\ a^{m_1} & a^{m_2} & \ldots & a^{m_n}\\ u_1 & u_2 & \ldots
    & u_n\end{array}\right)$. See more for the groups of this type
in~\cite{nek:bim,nek:fpresented}. It is not hard to show that $H_0(\G, \Z/2\Z)=\Z/2\Z$, where
image of $1_{vX^{\mathcal{N}}}$ is the generator of $\Z/2\Z$ for every
finite word $v$.

The element $\left(\begin{array}{ccc}1 & 2  & 3\\ a & a^{-1} & 1\\ 2 & 1 &
    3\end{array}\right)$ belongs to $\symm(\G)$, and its image in
$\symm(\G)/\alt(\G)$ is equal to the image of the generator of
$H_0(\G, \Z/2\Z)$.
The same is true for the element $\left(\begin{array}{ccc}1 & 2  & 3\\ 1 & 1 & 1\\ 2 & 1 &
    3\end{array}\right)$. It follows that $g=\left(\begin{array}{ccc}1 &
    2 & 3\\ a & a^{-1} & 1\\ 1 & 2 & 3\end{array}\right)$ belongs to
$\alt(\G)$.  The element $g$ is conjugate in $\symm(\G)$ to the
element
\[h=\left(\begin{array}{ccccc} 11 & 12 & 13 & 2 & 3\\ 1 & a & 1 & a^{-1} &
    1\\ 11 & 12 & 13 & 2 & 3\end{array}\right).\]
According to the recursive definition of $a$, the
element $g$ can be  represented by the table
\[g=\left(\begin{array}{ccccc} 11 & 12 & 13 & 2 & 3\\ 1 & a & 1 & a^{-1} &
    1\\ 12 & 11 & 13 & 2 & 3\end{array}\right).\]
It follows that
\[gh=\left(\begin{array}{ccccc} 11 & 12 & 13 & 2 & 3\\ 1 & 1 & 1 & 1 & 1
    \\ 12 & 11 & 13 & 2 & 3\end{array}\right)\]
belongs to $\alt(\G)$. But its image in $\symm(\G)/\alt(\G)$ must be equal to the
image of the generator of $H_0(\G, \Z/2\Z)$. It follows that
$\symm(\G)/\alt(\G)$ is trivial.

Other examples of the cases when the epimorphism $H_0(\G, \Z/2\Z)\arr\symm(\G)/\alt(\G)$ is not an isomorphism are described in~\cite{matui:product}.
\end{example}

\end{document}